\documentclass{IEEEtran}
\usepackage{cite}
\usepackage{url}
\usepackage{amsmath,amssymb,amsfonts}
\usepackage{multirow}
\usepackage{amsthm}
\usepackage{graphicx}
\usepackage{textcomp}
\usepackage{epsfig}
\usepackage{algorithm}
\usepackage{algorithmic}
\usepackage{epstopdf}
\usepackage{enumerate}
\usepackage{float}
\usepackage{graphics,graphicx}
\usepackage{subfigure}
\usepackage{array}
\usepackage{mathtools}
\usepackage{color}
\usepackage{mathtools}
\DeclareGraphicsExtensions{.png}
\graphicspath{{./Pics/}}
\usepackage{xspace}
\usepackage{tcolorbox}
\usepackage{dashbox}
\usepackage{bm} 
\usepackage{caption}
\usepackage{threeparttable}
\newtheorem{remark}{Remark}
\newtheorem{lemma}{Lemma}
\newtheorem{theorem}{Theorem}
\newtheorem{corollary}{Corollary}

\newtheorem{assumption}{Assumption}
\newtheorem{proposition}{Proposition}

\usepackage[justification=centering]{caption}
\begin{document}

	\title{\LARGE \bf Globally-Constrained Decentralized Optimization with \\Variable Coupling }
	\author{Dandan Wang, Xuyang Wu, Zichong Ou, and Jie Lu} 
	\author{Dandan Wang, Xuyang Wu, Zichong Ou, and Jie Lu
		\thanks{D. Wang and Z. Ou are with the School of Information Science and Technology, ShanghaiTech University, Shanghai 201210, China. Email: {\tt \{wangdd2, ouzch\}@shanghaitech.edu.cn}.}	
				\thanks{X. Wu is with the School of Automation and Intelligent Manufacturing, Southern University of Science and Technology, Shenzhen 518055, China, and the State Key Laboratory of Autonomous Intelligent Unmanned Systems, Beijing 100081, China. Email: {\tt wuxy6@sustech.edu.cn}.}
\thanks{Jie Lu is with the School of Information Science and Technology, ShanghaiTech University and the Shanghai Engineering Research Center of Energy Efficient and Custom AI IC, Shanghai 201210, China. Email:{\tt lujie@shanghaitech.edu.cn}.}
\thanks{This work was supported in part by the State Key Laboratory of Autonomous Intelligent Unmanned Systems under grant No. ZZKF2025-1-3.}
}

{}
\maketitle

\begin{abstract}
Many realistic decision-making problems in networked scenarios, such as formation control and collaborative task offloading, often involve complicatedly entangled local decisions, which, however, have not been sufficiently investigated yet. Motivated by this, we study a class of globally-constrained decentralized optimization problems with a variable coupling structure that is new to the literature. Specifically, we consider a network of nodes collaborating to minimize a global objective subject to a collection of global inequality and equality constraints, which are formed by the local objective and constraint functions of the nodes. On top of that, we allow such local functions to depend on not only the corresponding node's decision variable but the decisions of its neighbors as well. To address this problem, we propose a decentralized projected primal-dual algorithm, which incorporates gradient projection and virtual-queue techniques with a primal-dual-primal scheme. Under mild conditions, we derive $O(1/k)$ convergence rates for both objective error and constraint violations. Finally, two numerical experiments corroborate our theoretical results and illustrate the competitive performance of the proposed algorithm.
\end{abstract}



\section{Introduction}\label{sec:introduction}
Decentralized optimization has garnered considerable recent attention due to its broad applications in various networked systems such as communication networks\cite{Rabbat2004}, smart grids\cite{Magnusson2016}, and computing networks\cite{boyd2011distributed}. In these networked systems, each node possesses a collection of local data, typically inaccessible to other nodes for privacy protection, and all the nodes aim to make optimal decisions by solving a global optimization problem determined by all their local data. Decentralized optimization techniques allow the nodes to collaboratively solve such a problem through communicating with their neighboring nodes only, leading to high scalability with respect to the network size and the data volume.

To date, a host of decentralized optimization algorithms have been developed. Among these works, the majority focus on consensus optimization, which aims at finding a common decision that minimizes the sum of the local objectives under, mostly, no constraints \cite{Nedic2017,ShiW2015} or rather simple constraints such as set constraints \cite{WuXY2019,mai2019distributed}, local equality/inequality constraints\cite{yang2016multi, khatana2022dc}, and equality/inequality constraints known to all the nodes \cite{zhu2011distributed,yuan2015regularized}. Another line of research explores problems with local variables subject to coupling equality/inequality constraints \cite{Lakshmanan2008,Nedic2018,LinX2006,THChang2014,Falsone2017,Notarnicola2019,LiangS2019,falsone2023augmented,LiangS2019a,camisa2021distributed,LiuC2020,Mateos2016,wang2023distributed,wu2022distributed}. In particular, to solve the challenging problems with globally-coupled \textit{nonlinear inequality} constraints consisting of all the local constraint functions from the nodes, the distributed primal-dual subgradient algorithms \cite{THChang2014,Falsone2017, Mateos2016,LiuC2020,LiangS2019a,wang2023distributed,falsone2023augmented}, the distributed method based on operator splitting\cite{LiangS2019}, the distributed integrated primal-dual proximal algorithm \cite{wu2022distributed}, and the relaxation and successive distributed decomposition method\cite{Notarnicola2019} along with its extension \cite{camisa2021distributed} are developed. These works also derive various convergence results, including asymptotic convergence \cite{THChang2014,Falsone2017,Notarnicola2019,LiangS2019,falsone2023augmented} as well as $O(\ln k/\sqrt{k})$ \cite{LiangS2019a}, $O(1/\ln k)$ \cite{camisa2021distributed}, $O(1/\sqrt{k})$ \cite{LiuC2020,Mateos2016,wang2023distributed}, and $O(1/k)$ \cite{wu2022distributed} convergence rates.

A common limitation of the aforementioned decentralized optimization methods \cite{Nedic2017,ShiW2015,WuXY2019,mai2019distributed,Lakshmanan2008,Nedic2018,yang2016multi,khatana2022dc,zhu2011distributed,yuan2015regularized,LinX2006,THChang2014,Falsone2017,Notarnicola2019,LiangS2019,falsone2023augmented,LiangS2019a,camisa2021distributed,LiuC2020,Mateos2016,wang2023distributed,wu2022distributed} is that each node's local objective and constraint functions only rely on its own local decision variable. This excludes many realistic problems where each node's local function values are influenced by the decisions made by other nodes in the network, especially the node's neighbors. A representative example is the task offloading problem over cooperative mobile edge computing networks \cite{xiao2018distributed}, where each edge server's task completion time and local resource consumption are affected by the task offloaded from the neighboring servers. Such a variable coupling structure also arises in collaborative resource allocation \cite{huang2017collaborative}, distributed model predictive control \cite{WangCh2010}, wireless localization \cite{ShiW2010}, etc. 

A few existing works \cite{santilli2018finite,todescato2020partition,pu2016quantization,HuJH2018,XiaoYY2017,lu2024distributed,rai2024distributed,srivastava2021network,ran2023distributed,wang2024distributed,borst2013nonconcave,Mota2014,Alghunaim2019,li2021distributed,du2025distributed,alghunaim2021dual} allow for various variable coupling structures in local functions. However, they only focus on unconstrained problems\cite{santilli2018finite,todescato2020partition, li2021distributed}, problems with local constraint sets\cite{pu2016quantization}, locally-coupled constrained problems\cite{HuJH2018,XiaoYY2017,lu2024distributed,rai2024distributed,srivastava2021network,ran2023distributed,wang2024distributed,borst2013nonconcave, Mota2014, Alghunaim2019}, or globally-coupled affine constraints \cite{du2025distributed, alghunaim2021dual}, where \cite{XiaoYY2017,lu2024distributed,rai2024distributed,srivastava2021network,ran2023distributed} consider variable coupling in their local constraint functions but not in the local objectives. None of these works \cite{santilli2018finite,todescato2020partition,pu2016quantization,HuJH2018,XiaoYY2017,lu2024distributed,rai2024distributed,srivastava2021network,ran2023distributed,wang2024distributed,borst2013nonconcave,Mota2014,Alghunaim2019,li2021distributed,du2025distributed,alghunaim2021dual} consider globally-coupled nonlinear constraints like those in \cite{THChang2014,Falsone2017,Notarnicola2019,LiangS2019,falsone2023augmented,LiangS2019a,camisa2021distributed,LiuC2020,Mateos2016,wang2023distributed,wu2022distributed}.

Motivated by this, we consider a decentralized optimization problem with \textit{globally-coupled} nonlinear inequality and linear equality constraints. Moreover, for each node, we allow both the local objective function and the local constraint functions to involve \textit{variable coupling}, i.e., these local functions are associated with the decision variables of the node itself and all its neighbors. This coupling structure endows the optimization problem with a relatively general form that generalizes the prior prevalent decentralized optimization models in \cite{Nedic2017,ShiW2015,WuXY2019,mai2019distributed,yang2016multi,khatana2022dc,zhu2011distributed,yuan2015regularized,Lakshmanan2008,Nedic2018,LinX2006,THChang2014,Falsone2017,Notarnicola2019,LiangS2019,falsone2023augmented,LiangS2019a,camisa2021distributed,LiuC2020,Mateos2016,wang2023distributed,wu2022distributed} with no variable coupling in the local functions and the variable coupling models in \cite{santilli2018finite,todescato2020partition,pu2016quantization,HuJH2018,XiaoYY2017,lu2024distributed,rai2024distributed,srivastava2021network,ran2023distributed,wang2024distributed,borst2013nonconcave} with no constraints \cite{santilli2018finite,todescato2020partition} or with only local constraints \cite{pu2016quantization, HuJH2018,XiaoYY2017,lu2024distributed,rai2024distributed,srivastava2021network,ran2023distributed,wang2024distributed,borst2013nonconcave}.

In this paper, we develop a decentralized projected primal-dual algorithm for addressing the aforementioned globally-constrained optimization problem with variable coupling. To this end, we first construct a variant of the augmented Lagrangian function with respect to an equivalent form of the original problem, whose gradient has a certain degree of separability. Then, the primal update employs gradient projection to cause a descent of the augmented Lagrangian variant. The dual updates can be viewed as a non-trivial integration of ideas from the virtual queue method \cite{YuHao2017} and a dual variant of the P-EXTRA algorithm \cite{wu2019improved}. Eventually, the variable coupling structure and the globally-coupled constraints are effectively decoupled, so that our proposed algorithm can be executed in a fully decentralized fashion. We show that the proposed algorithm achieves $O(1/k)$ convergence rates in terms of both objective error and constraint violations under mild conditions. Its effectiveness and efficiency are demonstrated by two numerical examples.

The contributions of this paper are highlighted as follows:

\begin{itemize}
\item  To the best of our knowledge, this is the \textit{first} exploration of decentralized optimization with both variable coupling and globally-coupled constraints.
\item Compared to the existing works \cite{THChang2014,Falsone2017,Notarnicola2019,LiangS2019,falsone2023augmented,LiangS2019a,camisa2021distributed,LiuC2020,Mateos2016,wang2023distributed,wu2022distributed} that are able to solve problems with global inequality constraints but no variable coupling (i.e., special cases of our problem), the  $O(1/k)$ convergence rate of our proposed algorithm is stronger than the asymptotic convergence established in \cite{THChang2014,Falsone2017,Notarnicola2019,LiangS2019,falsone2023augmented}, faster than the convergence rates in \cite{LiangS2019a,LiuC2020,camisa2021distributed,Mateos2016,wang2023distributed}, and comparable to the rate in \cite{wu2022distributed}.

\item Our algorithm enjoys lower computational costs per iteration than many existing algorithms, because they require solving convex optimization problems at each iteration \cite{wu2022distributed,wang2023distributed,LiangS2019a,falsone2023augmented,camisa2021distributed,Notarnicola2019,Falsone2017,LiuC2020}, while our method only include simpler projection operations to update the primal variables. Additionally, our algorithm adopts constant step-sizes, which usually leads to faster convergence in practice compared to the algorithms with diminishing step-sizes \cite{THChang2014,Falsone2017,Notarnicola2019,LiangS2019a,camisa2021distributed,LiuC2020,Mateos2016,wang2023distributed}. 

\end{itemize}

This paper is organized as follows. Section~\ref{sec:probformulation} formulates the problem and presents examples of applications. Section~\ref{sec:algdevelop} develops the decentralized projected primal-dual algorithm and Section \ref{sec:convanal} provides the convergence analysis, followed by two numerical experiments in Section \ref{sec:example}. Finally, Section \ref{sec:conclusion} concludes the paper. 

\textbf{Notation and Definition:} For any set $X\subseteq\mathbb{R}^d$, $\operatorname{rel\;int} X$ is its relative interior. For any vectors $x,y\in \mathbb{R}^n$, we denote by $\max\{x, y\}$ the element-wise maximum of $x$ and $y$.  Let $\mathcal{S}\subseteq\{1,\ldots,n\}$. Given $n$ column vectors $x_1,\ldots,x_n$, $x_{\mathcal{S}}$ denotes the vector obtained by vertically stacking $x_i$ $\forall i\in\mathcal{S}$; given $n$ sets $X_1,\ldots,X_n$, $X_\mathcal{S}$ denotes the Cartesian product of $X_i$ $\forall i\in\mathcal{S}$. We use $\|\cdot\|$ to denote the Euclidean norm and $\{\cdot,\cdot\}$ to represent an unordered pair. The projection of $x$ on a convex set $X$ is denoted by $\mathcal{P}_{X}(x)$, i.e., $\mathcal{P}_{X}(x):=\mathop{\arg\min}_{x' \in X}\|x'-x\|$. In addition, $I_d$ and  $\mathbf{O}_d$ are the $d\times d$ identity matrix and all-zero matrix, respectively, and $\mathbf{1}_d$ ($\mathbf{0}_d$) is the $d$-dimensional all-one vector (all-zero vector), where the subscripts may be omitted for simplicity. For any  real matrix $A\in\mathbb{R}^{m\times n}$, $[A]_{ij}$ is its $(i,j)$-entry, $\operatorname{Range}(A)$ is its range, $\operatorname{Null}(A)$ is its null space, and $\|A\|$ is its spectral norm. If $A\in \mathbb{R}^{n\times n}$ is symmetric and positive semidefinite, $\|x\|_A:=\sqrt{x^TAx}$ for any $x\in \mathbb{R}^n$. 

For any differential function $f:\mathbb{R}^n\rightarrow\mathbb{R}^m$, $\operatorname{dom}(f):=\{x:~f(x)<+\infty\}$ is the domain of $f$; $\partial f(x)$  denotes the Jacobian matrix of function $f$ at $x$ and $\nabla f$ is the gradient of $f$ if $m=1$. A function $f$ is said to be \textit{$L_f$-smooth} over a set $X$ if there exists $L_f \ge 0$ such that $\|\nabla f(x)-\nabla f(y)\| \le L_f\|x-y\|$, $\forall x,y \in X$. Given $L \ge 0$, $f$ is said to be \textit{Lipschitz continuous} with Lipschitz constant $L$ over a set $X$ if  $\|f(x)-f(y)\| \le L\|x-y\|$, $\forall x,y \in X$.

\section{Problem Formulation}\label{sec:probformulation}

This section describes a decentralized globally-constrained optimization problem with variable coupling and presents a couple of application examples.

Consider a network modeled as an undirected, connected graph $\mathcal{G}=(\mathcal{V}, \mathcal{E})$, where the vertex set $\mathcal{V}=\{1, \ldots,n\}$ represents the set of $n$ nodes and the edge set $\mathcal{E}\subseteq\{\{i,j\}:i,j\in\mathcal{V},\;i\neq j\}$ represents the set of bidirectional links. Let $\mathcal{N}_{i}=\{j\in\mathcal{V}:\{i,j\} \in \mathcal{E}~or~i=j\}$ be the neighbor set of node $i$ including itself. 

Suppose all the nodes collaborate to solve the following global convex optimization problem:
\begin{equation} \label{eq:primalprob}
\begin{array}{cl}
	{\underset{x_i \in \mathbb{R}^{d_{i}},\;\forall i \in \mathcal{V} }{\operatorname{minimize}}} ~& {\sum\limits_{i \in \mathcal{V}} f_{i}\left(x_{\mathcal{N}_{i}}\right)} \\
	{\operatorname{ subject~to }} & {\sum\limits_{i \in \mathcal{V}} g_{i}\left(x_{\mathcal{N}_{i}}\right) \leq \mathbf{0}_p}, \\
	{}&{\sum\limits_{i \in \mathcal{V}}} A_ix_{\mathcal{N}_{i}} = {\sum\limits_{i \in \mathcal{V}}}b_i,\\
	{} & {x_i \in X_i,\;\forall i \in \mathcal{V}.}
\end{array}
\end{equation}
For each $i\in\mathcal{V}$, $x_i\in\mathbb{R}^{d_i}$ is the decision variable of node $i$ and is restricted to node $i$'s constraint set $X_i\subseteq\mathbb{R}^{d_i}$. Additionally, $f_i:\mathbb{R}^{\sum_{j\in\mathcal{N}_{i}}d_j}\rightarrow \mathbb{R}$ is the local objective function and $g_i:\mathbb{R}^{\sum_{j\in\mathcal{N}_{i}}d_j}\rightarrow \mathbb{R}^p$ is the local inequality constraint function of node $i$, which are known only to node $i$ but are associated with the decision variables of both node $i$ and its neighbors, denoted by $x_{\mathcal{N}_{i}}$. 
Each $g_i$ is a vector-valued function, and we express it as $g_i\left(x_{\mathcal{N}_{i}}\right)\!=\![g_{i1}\left(x_{\mathcal{N}_{i}}\right), \ldots, g_{ip}\left(x_{\mathcal{N}_{i}}\right)]^T$. Moreover, $A_i \!\in\! \mathbb{R}^{m \times \sum_{j \in \mathcal{N}_i}d_j}$ and $b_i \in \mathbb{R}^{m}$, locally owned by node $i$, encode node $i$'s local component of the global equality constraint in \eqref{eq:primalprob}, which, again, involves the decision variables of all the nodes in $\mathcal{N}_i$. 

Compared to most distributed optimization models\cite{Nedic2017,ShiW2015,WuXY2019,mai2019distributed,yang2016multi,khatana2022dc,zhu2011distributed,yuan2015regularized,Lakshmanan2008,Nedic2018,LinX2006,THChang2014,Falsone2017,Notarnicola2019,LiangS2019,falsone2023augmented,LiangS2019a,camisa2021distributed,LiuC2020,Mateos2016,wang2023distributed,wu2022distributed,santilli2018finite,todescato2020partition,pu2016quantization,HuJH2018,XiaoYY2017,lu2024distributed,rai2024distributed,srivastava2021network,ran2023distributed,wang2024distributed,borst2013nonconcave} in the literature, problem~\eqref{eq:primalprob} has a \textit{more complicated and general coupling structure} and, thus, is more challenging to solve. Specifically, the problems considered in many existing works \cite{Nedic2017,ShiW2015,WuXY2019,mai2019distributed,yang2016multi,khatana2022dc,zhu2011distributed,yuan2015regularized,Lakshmanan2008,Nedic2018,LinX2006,THChang2014,Falsone2017,Notarnicola2019,LiangS2019,falsone2023augmented,LiangS2019a,camisa2021distributed,LiuC2020,Mateos2016,wang2023distributed,wu2022distributed} are ``separable'' in the sense that each local function belonging to node $i$ is only associated with node $i$'s decision variable $x_i$. In contrast, problem~\eqref{eq:primalprob} is non-separable, as each $f_i$, $g_i$, and $A_i$ is associated with $x_{\mathcal{N}_{i}}$, i.e., $x_j$ $\forall j \in \mathcal{N}_i$. In addition, $g_i$, $A_i$, $b_i$ $\forall i \in \mathcal{V}$ constitute the globally-coupled inequality and equality constraints in an additive manner. Such a non-separable structure is not uncommon in networked applications, as each node's cost, utility, or capacity, etc., is often influenced by its neighbors' states and plays a role in the network-wide problems. Similar variable coupling structures are also considered in \cite{santilli2018finite,todescato2020partition,pu2016quantization,HuJH2018,XiaoYY2017,lu2024distributed,rai2024distributed,srivastava2021network,ran2023distributed,wang2024distributed,borst2013nonconcave}, yet \cite{santilli2018finite,todescato2020partition} only investigate unconstrained convex optimization, \cite{pu2016quantization} focus on problems with local constraint sets, and \cite{HuJH2018,XiaoYY2017,lu2024distributed,rai2024distributed,srivastava2021network,ran2023distributed,wang2024distributed,borst2013nonconcave} study locally-coupled constrained problems, where no variable coupling arises in the local objective functions in \cite{XiaoYY2017,lu2024distributed,rai2024distributed,srivastava2021network,ran2023distributed}. 

\begin{remark}\label{rmk:othervarcouplingmodels}
There have been several other coupling models for distributed convex optimization with an additive global objective function. In \cite{Mota2014} and \cite{Alghunaim2019}, each element in a global decision vector is assigned to a subset of nodes that are not necessarily neighbors, and the local objective of each node $i$ is a function of all the elements assigned to node $i$. In \cite{li2021distributed,du2025distributed}, each node's local objective involves not only its own decision variable but also a common aggregative term associated with the decision variables of all the nodes. The global objective in \cite{alghunaim2021dual} consists of a separable function and a known-to-all function whose variable is a linear combination of all nodes' decision variables. The above-mentioned objective functions are different from that of problem~\eqref{eq:primalprob}, while problem~\eqref{eq:primalprob} admits a more general constraint model than those of \cite{Mota2014,Alghunaim2019,li2021distributed,alghunaim2021dual,du2025distributed}. In particular, \cite{Mota2014, Alghunaim2019} only allow for local constraints, \cite{li2021distributed} only considers an unconstrained problem, \cite{du2025distributed,alghunaim2021dual} only work for affine constraints, yet \eqref{eq:primalprob} involves both nonlinear inequality and linear equality global constraints with coupled variables.

In fact, by letting each node maintain local copies of the decision variables of all the other nodes or local copies of all the neighbors' decision variables, problem~\eqref{eq:primalprob} can be converted into a globally-coupled consensus optimization problem or an optimization problem with additional local consensus constraints, so that a handful of existing algorithms are applicable (e.g.,\cite{li2020distributed,huang2023distributed,Alghunaim2019,wu2022distributed}). However, such reformulations would significantly increase the  variable dimension and, thus, the costs in memory, computations, and communications \cite{HuJH2018}. Therefore, we do not adopt such approaches.
\end{remark}
We impose the following assumption on problem~\eqref{eq:primalprob}.
\begin{assumption}\label{asm:primalprob}
Problem~\eqref{eq:primalprob} satisfies the following conditions:
\begin{enumerate}[(a)]
	\item For each $i\in\mathcal{V}$, $X_i$ is a compact convex set. \label{compactXi}
	\item For each $i\in\mathcal{V}$, $\sum_{i \in \mathcal{V}} f_{i}\left(x_{\mathcal{N}_{i}}\right)$ and $g_i$ are convex on the Cartesian products $X_{\mathcal{V}}$ and $X_{\mathcal{N}_{i}}$, respectively.\label{gisumficonvex} 
	\item For each $i\in\mathcal{V}$, $f_i$ is $L_f$-smooth on $X_{\mathcal{N}_{i}}$ for some $L_{f}\ge 0$.\label{fismooth}
	\item For each $i\in\mathcal{V}$, $g_{ij}$ $\forall j\in \{1,\ldots,p\}$ are $L_{g}$-smooth on $X_{\mathcal{N}_{i}}$ for some $L_{g}\ge 0$.\label{gijsmooth}
	\item For each $i\in\mathcal{V}$, $g_i$ is Lipschitz continuous on $X_{\mathcal{N}_{i}}$ with Lipschitz constant $\beta> 0$.\label{giLipschitz}
	\item There exists an optimum $\mathbf{x}^\star=[(x_1^\star)^T, \ldots, (x_n^\star)^T]^T\in X_{\mathcal{V}}$ to problem~\eqref{eq:primalprob}.\label{least1soluprimal}
	\item There exist $\tilde{x}_i\in\operatorname{rel\;int}(X_i)$ $\forall i\in\mathcal{V}$ such that ${\sum_{i \in \mathcal{V}} g_{i}\left( \tilde{x}_{\mathcal{N}_{i}}\right)<\mathbf{0}_p}$ and ${\sum_{i \in \mathcal{V}}} A_i\tilde{x}_{\mathcal{N}_{i}} = {\sum_{i \in \mathcal{V}}} b_i$.\label{slater}
\end{enumerate}
\end{assumption}

Assumption~\ref{asm:primalprob} indicates that problem~\eqref{eq:primalprob} is a smooth convex optimization problem with zero duality gap due to the Slater's condition in Assumption~\ref{asm:primalprob}(\ref{slater}). Note that we only assume convexity for the global objective function, so that each local objective $f_i$ is not necessarily convex. Also, the Lipschitz conditions of each $\nabla f_i$, $\nabla g_{ij}$, and $g_i$ in Assumption~\ref{asm:primalprob}(\ref{fismooth}), \ref{asm:primalprob}(\ref{gijsmooth}), and \ref{asm:primalprob}(\ref{giLipschitz}) only need to hold on the compact constraint set $X_{\mathcal{N}_{i}}$. The conditions in Assumption~\ref{asm:primalprob} are mild and commonly adopted in existing works\cite{THChang2014,LiangS2019,wang2024distributed}. Many practical application problems satisfy the conditions in Assumption~1, including the two motivating examples in Section~\ref{sec:probformulation}-A. We will further compare these conditions with the assumptions adopted in related works in Section~\ref{ssec:compdiss}.

\subsection{Examples of Applications}
Many real-world engineering problems can be cast into the form of \eqref{eq:primalprob}. Below, we provide two such examples.

\subsubsection{A Multi-Vehicle Formation Stabilization Problem}
Consider a multi-vehicle formation stabilization problem. Our goal is to stabilize $n$ vehicles on a plane toward an equilibrium point in a decentralized way. For each vehicle $i$, $s_i^k=[(p_i^k)^T,(v_i^k)^T]^T\in\mathbb{R}^{4}$ is the state and $u_i^k\in\mathbb{R}^{2}$ is the control input at time $k\ge0$, where $p_i^k\in\mathbb{R}^{2}$, $v_i^k\in\mathbb{R}^{2}$ are the position and the velocity of vehicle $i$, respectively. The system dynamics of each vehicle $i$ is given by
\begin{equation}
s_{i}^{k+1}=A_{i} s_{i}^{k}+B_{i} u_{i}^{k}.\label{eq:stateeq}
\end{equation}
Given the initial states $s_i^0$ $\forall i=1,\ldots,n$, all the
vehicles obtain their control actions through model predictive control (MPC), which requires them to cooperatively solve a constrained convex optimization problem. The global cost is $L(s,u)=\sum_{i=1}^{n}\sum_{k=0}^{N-1}\alpha_1\sum_{j\in \mathcal{N}_i}\|p_i^k-p_j^k+d_{ij}\|^2+\alpha_2\|v_i^k\|^2+\alpha_3\|u_i^k\|^2+L^d_i$, where $N\ge1$ is the prediction horizon in MPC, $\alpha_i \forall i={1,2,3}$ are positive weighting constants, and $d_{ij}$ is the desired equilibrium position between node $i$ and $j$. The tracking cost $L^d_i$ is defined as $\alpha_4\|\sum_{i=1}^{m}p_i^k/m\!-\!p_d\|^2$ ($m<n$) if $i$ is one of the $m$ core vehicles; otherwise, $L^d_i=0$  \cite{dunbar2006distributed}. Let the optimization variable associated with each vehicle $i$ be $x_i=[(s_i^1)^T, \ldots, (s_i^{N-1})^T, (u_i^0)^T, \ldots, (u_i^{N-1})^T]^T\in\mathbb{R}^{6N}$. The multi-vehicle formation stabilization problem is formulated as the following optimization problem 
\begin{equation} \label{eq:coupledmpc}
\begin{array}{cl}
	{\underset{x_i,\;\forall i=1,\ldots,n}{\operatorname{minimize}}} ~& \sum\limits_{i=1}^{n} \Big(x_{\mathcal{N}_{i}}^TH_i x_{\mathcal{N}_{i}}+Q_i^Tx_{\mathcal{N}_{i}}\Big)  \\
	\operatorname{ subject~to } & \sum_{i=1}^{n}P_ix_{i}\le b_i,\\
	&C_ix_{i}=b_i',\;\forall i=1,\ldots,n,\\
	& x_i \in X_i,\;\forall i=1,\ldots,n.
\end{array}
\end{equation}
In \eqref{eq:coupledmpc}, each $H_i$ is a symmetric positive semidefinite matrix and each $Q_i \in\mathbb{R}^{6| \mathcal{N}_{i}|N}$, which can be obtained from the definition of $x_i$ and $L(s,u)$. In \eqref{eq:coupledmpc}, each vehicle $i$’s local cost is influenced by the positions of its neighboring vehicles \cite{dunbar2006distributed}. It can be verified that problem~\eqref{eq:coupledmpc} satisfies Assumption~1. The first constraint in \eqref{eq:coupledmpc} stands for resource limitations and position/velocity requirements across the whole system. The second constraint represents the dynamics of each vehicle $i$ over the prediction horizon, which is directly derived from \eqref{eq:stateeq}. The last constraint restricts each $x_i$ to some compact convex set $X_i$ due to certain system limitations. In \eqref{eq:coupledmpc}, $H_i$, $Q_i$,  $P_i$, $b_i$, $C_i$, $b_i'$, and $X_i$ are accessible only to vehicle $i$.  Clearly, the distributed MPC problem~\eqref{eq:coupledmpc} is a particular form of \eqref{eq:primalprob}.

\subsubsection{Task Offloading over Cooperative Mobile Edge Computing Networks} 
Consider a task offloading problem among $n$ edge servers forming a connected network. Each server receives a large number of latency-sensitive tasks from exogenous sources (e.g., mobile devices) and, due to limited and heterogeneous computation resources, needs to offload part of its tasks to neighboring servers to make full use of computation resources\cite{xiao2018distributed}. All the edge servers cooperate to find optimal task offloading decisions by minimizing the total completion time of all the tasks subject to the limitations of channel bandwidth or computation capability.

Let $\mathcal{V}=\{1,2,\ldots, n\}$ denote the set of edge servers, and $\mathcal{N}_{i}$ the neighbor set of edge server $i$ including itself. Define $x_{ii}$ and $x_{ij}$ as the quantity of tasks allocated by edge server $i$ to itself and to its neighboring edge server $j$, respectively, and $x_i=[x_{ij}]_{\forall j \in \mathcal{N}_i}$ as the task offloading decision of edge server $i$. Consider an $M/M/1$ queuing system for each edge server to process its received tasks, where the task arrival $\lambda_i$ per unit time follows a Poisson distribution and $\mu_{i}$ is the maximum task processing rate \cite{wang2024distributed}. According to the queuing theory, the processing latency of edge server $i$ is $f_i\left(x_{\mathcal{N}_{i}}\right)=\frac{\sum_{j \in \mathcal{N}_i}\lambda_jx_{ji}}{\mu_i-\sum_{j \in \mathcal{N }_i}\lambda_jx_{ji}}$ and the computing capacity constraint is $g_i\left(x_{\mathcal{N}_{i}}\right)=\sum_{j \in \mathcal{N}_i} \lambda_{j}x_{ji}-\mu_{i}$, where both $f_i$ and $g_i$ are determined by the tasks offloaded by both edge server $i$ itself and its neighbors, i.e., $x_{ji}$ $\forall j \in \mathcal{N}_i$. For each edge server, the sum of the task offloaded to itself and neighboring servers should be equal to the tasks it received from the exogenous sources, leading to linear equality constraints. Therefore, such a task offloading problem in \cite{wang2024distributed} is in the form of \eqref{eq:primalprob}.

\section{Algorithm Development}\label{sec:algdevelop}

This section develops a decentralized algorithm to solve problem \eqref{eq:primalprob}.

\subsection{Problem Transformation}\label{ssec:probtran}
We first transfer \eqref{eq:primalprob} into an equivalent form to facilitate our algorithm design.

By adding $n$ auxiliary variables $t_i\in\mathbb{R}^p$ $\forall i\in\mathcal{V}$, the inequality constraint $\sum_{i\in \mathcal{V}} g_i(x_{\mathcal{N}_{i}})\le \mathbf{0}_p$ can be converted to $g_i(x_{\mathcal{N}_{i}})\le t_i$ $\forall i\in\mathcal{V}$ and $\sum_{i\in \mathcal{V}} t_i=\mathbf{0}_p$. Let 
\begin{align}
		y_i&=[x_i^T, t_i^T]^T, \forall i\in \mathcal{V}, \displaybreak[0]\\
		\mathbf{y}&=[y_1^T, \ldots, y_n^T]^T\in \mathbb{R}^{N},
\end{align}
where $N=np+\sum_{i\in \mathcal{V}} d_i$. Then, problem~\eqref{eq:primalprob} can be equivalently transformed into 
\begin{equation}\label{eq:probleminy}
\begin{array}{cl}
	\underset{\mathbf{y}\in \mathbb{R}^{N}}{\operatorname{minimize}} ~& \mathbf{f}(\mathbf{y}):=\sum\limits_{i\in \mathcal{V}} f_i(x_{\mathcal{N}_{i} })\\
	\operatorname{subject~to} &\mathbf{G}(\mathbf{y}):= \left[\begin{array}{c}
		g_1(x_{\mathcal{N}_{1} })\!-\!t_1\\
		\vdots\\
		g_n(x_{\mathcal{N}_{n}})\!-\!t_n
	\end{array}\right]\! \le\! \mathbf{0}_{np},\\
	& (\mathbf{1}_n\otimes I_{m+p})^T(\mathbf{B}\mathbf{y}-\mathbf{c})=\mathbf{0}_{m+p},\\
	&\mathbf{y}\in Y,
\end{array}
\end{equation}
where $Y$ is the Cartesian product of $Y_i=\{[x_i^T, t_i^T]^T: x_i \in X_i, \ t_i \in \mathbb{R}^p\}$ $\forall i\in \mathcal{V}$, $\mathbf{B}=\operatorname{diag}(B_1,\ldots,B_n)$ with $B_i=\operatorname{diag}(\sum_{j \in \mathcal{N}_i}A_{ji}, I_p)$, and $\mathbf{c}=[c_1^T,\ldots,c_n^T]^T$ with $c_i=[b_i^T, \mathbf{0}_p^T]^T$. Here, we partition each $A_i$ into $A_{ij} \in \mathbb{R}^{m \times d_j}$ $\forall j \in \mathcal{N}_{i}$ with each $A_{ij}$ formed by $d_j$ columns of $A_i$.

Assumption~\ref{asm:primalprob} on problem~\eqref{eq:primalprob} leads to the following properties of the equivalent problem \eqref{eq:probleminy}.

\begin{proposition}\label{pro:propertyofproby}
Suppose Assumption~\ref{asm:primalprob} holds. Then, problem~\eqref{eq:probleminy} satisfies the following properties:
\begin{enumerate}[(a)]
	\item $\mathbf{f}(\mathbf{y})$ and $\mathbf{G}(\mathbf{y})$ are convex on the set $Y$.\label{pro:fygyconvex}
	\item $\nabla \mathbf{f}$ is Lipschitz continuous with Lipschitz constant $L_F:=L_f\max_{i\in \mathcal{V}}(|\mathcal{N}_{i}|)$ on the set $Y$.\label{pro:gardfsmooth}
	\item $\mathbf{G}(\mathbf{y})$ is Lipschitz continuous with Lipschitz constant  
	$\tilde{\beta}:=\sqrt{(1+\beta^2)(\max_{i\in \mathcal{V}} (|\mathcal{N}_i|))}$ on the set $Y$. \label{pro:Gycontinuous}
	\item There exists at least one optimum $\mathbf{y}^\star$ to problem~\eqref{eq:probleminy}. Any optimal solution $\mathbf{y}^\star = [(x_1^\star)^T, (t_1^\star)^T, \ldots, (x_n^\star)^T, (t_n^\star)^T]^T$ to problem~\eqref{eq:probleminy} satisfies that 
	$[(x_1^\star)^T, \ldots, (x_n^\star)^T]^T$ is an optimal solution to problem~\eqref{eq:primalprob}. \label{pro: equivsolguarteen}
	\item Strong duality holds between problem~\eqref{eq:probleminy} and its Lagrange dual problem. \label{pro:equivstrongdual} 
\end{enumerate}
\end{proposition}

\begin{proof}
See Appendix \ref{ssec:proofofproLipschitz}.
\end{proof}
Proposition~\ref{pro:propertyofproby}(\ref{pro: equivsolguarteen}) shows that problem~\eqref{eq:probleminy} and problem~\eqref{eq:primalprob} are equivalent, which allows us to obtain the optimal solution of problem~\eqref{eq:primalprob} by solving problem~\eqref{eq:probleminy}.

\subsection{Algorithm Design}
This subsection designs an algorithm for solving the equivalent problem~\eqref{eq:probleminy}.

Our proposed algorithm maintains four variables $\mathbf{y}^{k}\in \mathbb{R}^{N}$, $\mathbf{q}^{k}\in \mathbb{R}^{np}$, $\mathbf{u}^{k}\in \mathbb{R}^{n(m+p)}$, and $\mathbf{z}^{k}\in \mathbb{R}^{n(m+p)}$, where $\mathbf{y}^{k}$ is the primal variable, $\mathbf{q}^{k}$ and $\mathbf{u}^{k}$ are the dual variables associated with the inequality constraint and the equality constraint in \eqref{eq:probleminy}, respectively, and $\mathbf{z}^{k}$ is an auxiliary variable whose role will be explained shortly.

The algorithm parameters include two constant step-sizes $\gamma, \rho>0$ and two weight matrices $W, H$ that facilitate information fusion across the network and comply with the network topology. We let $W=P^W\otimes I_{m+p},~H=P^H\otimes I_{m+p}$, and require $P^W,P^H\in \mathbb{R}^{n\times n}$ meet the assumption below.
\begin{assumption}\label{asm:matricesphpw}
(a) $[P^W]_{ij}=[P^H]_{ij}=0$, $\forall i \in \mathcal{V}$, $\forall j\notin \mathcal{N}_i $; (b) $P^W$ and $P^H$ are symmetric and positive semidefinite; (c) $P^W\mathbf{1}_n=\mathbf{1}_n$, $\operatorname{Null}(P^H)=\operatorname{span}(\mathbf{1}_n)$; (d) $P^W+P^H \preceq I_n$.
\end{assumption}

Assumption~\ref{asm:matricesphpw}(a) is used to enable distributed implementation, while (b)–(d) are imposed to ensure convergence; see \cite{wu2022distributed,ShiW2015a,wu2019improved} for detailed discussions. There are many options for $P^W,P^H$ under Assumption~\ref{asm:matricesphpw}. For example, we may let $P^W\!=\!\frac{I+P'}{2}$ and $P^H\!=\!\frac{I-P'}{2}$ for some $P' \in \mathbb{R}^{n \times n}$ such that $[P']_{ij}=[P']_{ji}>0$ $ \forall \{i,j\}\in \mathcal{E}$, $\sum_{j\in \mathcal{N}_{i}}[P']_{ij}=1$, $[P']_{ii}>0$, and $[P']_{ij}=0 $ otherwise \cite{wu2022distributed}.

At initial iteration $k=0$, we arbitrarily select $\mathbf{y}^{0}\in Y$ and $\mathbf{u}^0\in\mathbb{R}^{n(m+p)}$. Then, we initialize $ \mathbf{z}^{0}=\rho H\mathbf{u}^{0}$ and $\mathbf{q}^{0} = \max\{-\mathbf{G}(\mathbf{y}^{0}),\mathbf{0}_{np}\}$. At each subsequent iteration $k+1$, all the variables are updated according to
\begin{align}
\mathbf{y}^{k+1}&=\mathcal{P}_{ Y}[\mathbf{y}^{k}-\gamma\cdot \mathbf{d}^{k}],\displaybreak[0]\label{eq:projofyk} \\
\mathbf{q}^{k+1}& = \max\{-\mathbf{G}(\mathbf{y}^{k+1}), \mathbf{q}^{k}\!+\!\mathbf{G}(\mathbf{y}^{k+1})\},\displaybreak[0]\label{eq:updateofqk}\\
\mathbf{u}^{k+1}&=W \mathbf{u}^{k}+\frac{1}{\rho}(\mathbf{B} \mathbf{y}^{k+1}-\mathbf{c}-\mathbf{z}^{k}),\displaybreak[0]\label{eq:updateofuk}\\
\mathbf{z}^{k+1} &=\mathbf{z}^{k}+\rho H\mathbf{u}^{k+1}.\label{eq:updateofzk}
\end{align}

In the primal update \eqref{eq:projofyk}, $\mathbf{d}^{k}$ is the gradient of the following augmented-Lagrangian-like function $\mathcal{R}^k(\mathbf{y})$ at $\mathbf{y}^k$:
\begin{align}
\mathcal{R}^k(\mathbf{y})=&\mathbf{f}(\mathbf{y})+\langle \mathbf{q}^{k}+\mathbf{G}(\mathbf{y}^{k}), \mathbf{G}(\mathbf{y})\rangle\nonumber\displaybreak[0]\\
&+\langle W\mathbf{u}^{k}-\frac{1}{\rho}\mathbf{z}^{k}, \mathbf{B}\mathbf{y}-\mathbf{c}\rangle+\frac{1}{2\rho}\|\mathbf{B}\mathbf{y}-\mathbf{c}\|^2,
\label{eq:defineofRt}
\end{align}
and the expression of $\mathbf{d}^{k}$ is thus given by
\begin{align}
\mathbf{d}^{k}=&\nabla \mathbf{f}(\mathbf{y}^{k}) +\mathbf{B}^T(W\mathbf{u}^{k}-\frac{1}{\rho}\mathbf{z}^{k})+\frac{1}{\rho}\mathbf{B}^{T}\mathbf{B}\mathbf{y}^{k}\nonumber\displaybreak[0]\\
&+(\frac{\partial\mathbf{G}(\mathbf{y}^{k})}{\partial  \mathbf{y}})^T(\mathbf{q}^{k}+\mathbf{G}(\mathbf{y}^{k}))-\frac{1}{\rho}\mathbf{B}^T\mathbf{c},\label{eq:jocabianofr} 
\end{align}
where $\frac{\partial\mathbf{G}(\mathbf{y}^{k})}{\partial \mathbf{y}}$ is the Jacobian matrix of $\mathbf{G}(\mathbf{y})$ at $\mathbf{y}^{k}$. Clearly, the primal update \eqref{eq:projofyk} is a gradient projection operation aiming at the descent of the augmented-Lagrangian-like function $\mathcal{R}^k(\mathbf{y})$ in \eqref{eq:defineofRt}.

The dual update \eqref{eq:updateofqk} adopts the idea of virtual queue in \cite{YuHao2017}, which was originally developed for addressing inequality-constrained convex problems. It can be viewed as a modification of the dual subgradient algorithm with an additional proximal term in the primal update and an approximate dual update that resembles a queuing equation. In the absence of the equality constraints in \eqref{eq:probleminy}, minimizing $\mathcal{R}^k(\mathbf{y})$ plus a proximal term $\frac{\gamma}{2}\|\mathbf{y}-\mathbf{y}^{k}\|^2$ corresponds to the primal update of the virtual queue algorithm in \cite{YuHao2017}, while \eqref{eq:updateofqk} matches its dual update. However, this primal update cannot be implemented in a distributed manner because of the variable coupling in the local functions. This motivates us to employ a different primal update \eqref{eq:projofyk} for distributed implementation, and retain the dual update \eqref{eq:updateofqk} for effectively handling the inequality constraints. 
 
To understand the rationale behind the remaining dual updates \eqref{eq:updateofuk} and \eqref{eq:updateofzk}, we consider the dual problem of \eqref{eq:probleminy} in absence of the inequality constraint $\mathbf{G}(\mathbf{y})\le\mathbf{0}$\footnote{The Lagrange dual problem is originally given by $\operatorname{maximize}_{u\in \mathbb{R}^{m+p}}\min_{\mathbf{y}\in\mathbb{R}^N} \mathbf{f}(\mathbf{y})+\langle u, (\mathbf{1}_n\otimes I_{m+p})^T(\mathbf{B}\mathbf{y}-\mathbf{c})\rangle$. Here we consider its equivalent form \eqref{eq:equalitydualproblem} for our algorithm design.}:
\begin{equation}\label{eq:equalitydualproblem}
\begin{split}
	\underset{\mathbf{u}\in \mathbb{R}^{n(m+p)}}{\operatorname{maximize}}~~&~ D_e(\mathbf{u}):=\min_{\mathbf{y}\in Y} \mathbf{f}(\mathbf{y})+\langle\mathbf{u},\mathbf{B}\mathbf{y}-\mathbf{c}\rangle\\
	\operatorname{subject~to} ~&~ ~u_1 = \ldots = u_n,
\end{split}
\end{equation}
where $\mathbf{u}=[u_1^T,\ldots,u_n^T]^T$.
We first make an attempt to solve the dual problem \eqref{eq:equalitydualproblem} by employing the distributed first-order method P-EXTRA \cite{ShiW2015a}. According to \cite{wu2019improved}, such a dual P-EXTRA algorithm gives
\begin{align}
\mathbf{u}^{k+1}=& \operatorname{\arg\;\min}_{\mathbf{u}\in \mathbb{R}^{n(m+p)}} -D_e(\mathbf{u})+\langle\mathbf{u}, \mathbf{z}^{k}\rangle\nonumber\displaybreak[0]\\
&+\frac{\rho}{2}\|\mathbf{u}-W\mathbf{u}^{k}\|^2,\label{eq:pextrauupdate}
\end{align}
where $\mathbf{z}^k$ takes the same recursive form as \eqref{eq:updateofzk} with $\mathbf{z}^0=\rho H\mathbf{u}^0$. Here, $\mathbf{z}^k$ is the dual variable associated with the consensus constraint $u_1 = \ldots = u_n$ in problem \eqref{eq:equalitydualproblem}. According to the update equation \eqref{eq:updateofzk} with $ \mathbf{z}^{0}=\rho H\mathbf{u}^{0}$, we have $\mathbf{z}^{k}=\rho \sum_{\ell=0}^{k}H\mathbf{u}^{\ell}$, which is the accumulation of the consensus error over iterations and facilitates the consensus of all the $u_i^\ell$'s. Note that the above updates are not implementable because in general the dual function $D_e(\mathbf{u})$ is inaccessible.

To overcome this issue, we propose the following primal-dual-primal approach: First, note that we may rewrite \eqref{eq:pextrauupdate} as
\begin{align*}
\mathbf{z}^k+\rho(\mathbf{u}^{k+1}-W\mathbf{u}^k)\in \partial D_e(\mathbf{u}^{k+1}).
\end{align*}
On the other hand, if we set
\begin{equation}\label{eq:yk1isminimizer}
\mathbf{y}^{k+1}\in \operatorname{\arg\;\min}_{\mathbf{y}\in Y} \mathbf{f}(\mathbf{y})+\langle \mathbf{u}^{k+1}, \mathbf{B}\mathbf{y}-\mathbf{c}\rangle,
\end{equation}
then $\mathbf{B}\mathbf{y}^{k+1}-\mathbf{c}$ is also a subgradient of $ D_e(\mathbf{u}^{k+1})$ like $\mathbf{z}^k+\rho(\mathbf{u}^{k+1}-W\mathbf{u}^k)$. Hence, we choose to equalize these two items, leading to the update equation of $\mathbf{u}^{k+1}$ which is exactly \eqref{eq:updateofuk}. Then, to derive a legitimate update equation for $\mathbf{y}^{k}$, note that \eqref{eq:yk1isminimizer} can be rewritten as $-\mathbf{B}^T\mathbf{u}^{k+1}\in \nabla \mathbf{f}(\mathbf{y}^{k+1})$, which, along with \eqref{eq:updateofuk}, results in $-\mathbf{B}^T(W\mathbf{u}^k+\frac{1}{\rho}(\mathbf{B}\mathbf{y}^{k+1}-\mathbf{c}-\mathbf{z}^{k}))\in \nabla \mathbf{f}(\mathbf{y}^{k+1})$. This is equivalent to 
\begin{align}
\mathbf{y}^{k+1}\in\operatorname{\arg\;\min}_{\mathbf{y}\in Y}\mathcal{R}^k(\mathbf{y}),\label{eq:yinargminR}
\end{align}
where $\mathcal{R}^k(\mathbf{y})$ is given in \eqref{eq:defineofRt} (with $\mathbf{G}(\mathbf{y})\equiv\mathbf{0}$). Therefore, \eqref{eq:yinargminR}, \eqref{eq:updateofuk}, and \eqref{eq:updateofzk} together constitute an equivalent form of the aforementioned dual P-EXTRA algorithm, i.e., \eqref{eq:pextrauupdate} and \eqref{eq:updateofzk}. Recall that \eqref{eq:projofyk} is a gradient projection operation for emulating \eqref{eq:yinargminR}. Consequently, for problem \eqref{eq:probleminy} without the inequality constraint, our proposed algorithm can be viewed as a dual first-order method that originates from dual P-EXTRA yet introduces the gradient projection operation to acquire readily available and computationally efficient estimates of dual subgradients.

\begin{remark}\label{rmk:comaprsionWucoupling}
Our proposed algorithm does not directly minimize $\mathcal{R}^k(\mathbf{y})$, i.e., adopt \eqref{eq:yinargminR} as the primal update. This is because $f_i$ and $g_i$ appearing in $\mathcal{R}^k(\mathbf{y})$ are functions of coupled variables, which disable the decentralized realization of \eqref{eq:yinargminR}. Similarly, the primal update of IPLUX in \cite{wu2022distributed} minimizes $\mathcal{R}^k(\mathbf{y})$ with $\mathbf{f}(\mathbf{y})$ substituted by its first-order approximation, which can only tackle problems without variable coupling and, thus, is unable to solve problem~\eqref{eq:primalprob} without any centralized coordination. In contrast, the gradient of $\mathcal{R}^k(\mathbf{y})$ has a decomposable structure, so that the gradient projection step \eqref{eq:projofyk} can be executed decentralizedly, which will be elaborated in Section~\ref{ssec:distributed}. Also, \eqref{eq:projofyk} can significantly reduce the computational complexity of \eqref{eq:yinargminR}. Nevertheless, by adopting \eqref{eq:projofyk} instead of \eqref{eq:yinargminR}, we indeed compromise on the accuracy in estimating dual subgradients, thereby intensifying the challenge of convergence analysis. 
\end{remark}

\subsection{Decentralized Implementation}\label{ssec:distributed}
Below we describe the decentralized implementation of the proposed algorithm \eqref{eq:projofyk}--\eqref{eq:updateofzk}.

Let each node $i \in \mathcal{V}$ maintain $y_i^k\in \mathbb{R}^{d_i+p}, q_i^k\in \mathbb{R}^{p},u_i^k \in \mathbb{R}^{m+p}$, and $z_i^k \in\mathbb{R}^{m+p}$, where $y_i^k, q_i^k,u_i^k, z_i^k$ are the $i$-th block of the variables $\mathbf{y}^k,\mathbf{q}^k,\mathbf{u}^k,\mathbf{z}^k$, respectively. We further partition $y_i^k=[(x_i^k)^T, (t_i^k)^T]^T$, where $x_i^k \in\mathbb{R}^{d_i}$ and $t_i^k \in \mathbb{R}^{p}$.

To obtain the update of $y_i^k$, let $\mathbf{d}^{k}=[(d_1^{k})^T, \ldots, (d_n^{k})^T]^T$, where $d_i^{k}=\nabla_{y_i}\mathcal{R}^k(\mathbf{y}^{k})\! \in \!\mathbb{R}^{d_i+p}$ $\forall i \in \mathcal{V}$ represents the gradient of $\mathcal{R}^k(\mathbf{y})$ with respect to $y_i^k$. The expression of $d_i^k$ $\forall i \in \mathcal{V}$ directly follows from \eqref{eq:jocabianofr} and needs to compute $\nabla_{y_i}\mathcal{R}^k(\mathbf{y}^{k})$ by utilizing only elementary matrix operations, which is given
\begin{align}
	d_i^k=B_i^T\Big(\!\sum\limits_{j \in \mathcal{N}_{i}}\! [P^W]_{i j} u_{j}^k-\frac{1}{\rho}z_i^k \Big)+[(\tilde{d}_{i1}^{k})^T, ~(\tilde{d}_{i2}^{k})^T]^T, \displaybreak[0]\label{eq:jacobianofyi} 
\end{align}
where
\begin{align}
\tilde{d}_{i1}^{k}&= \!\sum\limits_{j \in \mathcal{N}_{i}}\!\! \Big(\nabla_{x_i} f_j(x^k_{\mathcal{N}_{j} })+(\frac{\partial g_j(x^k_{\mathcal{N}_{j} })}{\partial x_i})^T(q_j^k+ g_j(x_{\mathcal{N}_{j} }^k)-t_j^k)\!\Big)\notag \\
&~~~~~+\frac{1}{\rho}(\sum\limits_{j \in \mathcal{N}_i}A_{ji})^T\!\sum\limits_{j \in \mathcal{N}_i}\!\!\!A_{ji}x_i^k-\frac{1}{\rho}(\sum\limits_{j \in \mathcal{N}_i}A_{ji})^Tb_i,\label{eq:expressionofdi1}\\
\begin{split}
	\tilde{d}_{i2}^{k}&=\frac{1}{\rho}t_i^k-(q_i^k+ g_i(x_{\mathcal{N}_{i} }^k)-t_i^k).\label{eq:expressionofdi2}
\end{split}
\end{align}
The above equations \eqref{eq:jacobianofyi}--\eqref{eq:expressionofdi2} indicate that $d_i^k$ can be obtained by information interaction between neighboring nodes.

By utilizing the structure of the projection operation, $\mathbf{y}^{k+1}$ in \eqref{eq:projofyk} is decomposed as the following local updates:
\begin{equation}\label{eq:updateofyi}
y_i^{k+1}=\mathcal{P}_{Y_i} [y_i^{k}- \gamma \cdot d_i^{k}] ~~\forall  i \!\in\!\mathcal{V}.
\end{equation}
Owing to the neighbor sparse structures of $W, H$ and the definition of $\mathbf{G}$ in \eqref{eq:probleminy}, the decentralized updates of $\mathbf{q}^{k+1}$, $\mathbf{u}^{k+1}$, and $\mathbf{z}^{k+1}$ in \eqref{eq:updateofqk}--\eqref{eq:updateofzk} can be easily rewritten as the following local operations: For each $k\ge 0$ and each $i \in \mathcal{V}$,  
\begin{align}
q_{i}^{k+1}&=\max\left\{t_i^{k+1}\!-g_i(x_{\mathcal{N}_{i} }^{k+1}\!), ~q_{i}^{k}\!+\!g_i(x_{\mathcal{N}_{i}}^{k+1}\!)\!-\!t_i^{k+1}\right\},\displaybreak[0]\label{eq:updateofQi}\\
u_{i}^{k+1}&=\sum_{j \in \mathcal{N}_{i} } [P^W]_{i j} u_{j}^{k}+\frac{1}{\rho}\left(B_{i} y_{i}^{k+1}-c_i-z_{i}^{k}\right),\displaybreak[0]\label{eq:updateofui}\\
z_{i}^{k+1}&=z_{i}^{k}+\rho \sum_{j \in \mathcal{N}_{i} }[P^H]_{i j} u_{j}^{k+1}.\label{eq:updateofzi}
\end{align}

Algorithm~\ref{alg:algorithm} describes the detailed implementation of \eqref{eq:updateofyi}--\eqref{eq:updateofzi} taken by all the nodes. Given arbitrary $y_i^0 \in Y_i$, $ u_i^0 \in \mathbb{R}^{p}$ $\forall i \in \mathcal{V}$, each node $i \!\in \! \mathcal{V}$ calculates $z_i^0, q_i^0$ by collecting $u_j^0, x_j^0$ from its every neighbor $ j \in \mathcal{N}_i$. Then, at each $k \ge 0$, each node $i \in\!\!\mathcal{V}$ updates its primal variable $y_i^{k+1}$ according to \eqref{eq:updateofyi}, which requires node $i$ to calculate $d_i^k$ first. According to \eqref{eq:jacobianofyi}--\eqref{eq:expressionofdi2},
each node $i$ needs to collect  $x_j^k$, $u_j^k$, $ A_{ji}$, and $ \nabla_{x_i} f_j(x^k_{\mathcal{N}_{j} })+(\frac{\partial g_j(x^k_{\mathcal{N}_{j} })}{\partial x_i})^T(q_j^k+ g_j(x_{\mathcal{N}_{j} }^k)-t_j^k)$ from every neighbor node $j$ for obtaining $\tilde{d}_{i1}^{k}, \tilde{d}_{i2}^{k}$ and then computes $d_i^k$ based on them.  
The updates of $q_i^{k+1}, u_i^{k+1}$ in \eqref{eq:updateofQi}, \eqref{eq:updateofui} can be realized by requesting  $x_j^{k+1}$ from neighbor $j$. In addition, each node $i\in\mathcal{V}$ computes $z_i^{k+1}$ according to \eqref{eq:updateofzi} by collecting $u_j^{k+1}$ from its neighbor $j$. Obviously, the updates \eqref{eq:updateofyi}--\eqref{eq:updateofzi} can be realized in a decentralized and inexpensive way.

{\renewcommand{\baselinestretch}{1.05}
\begin{algorithm} [t] 
	\caption{ Decentralized Projected Primal-Dual Algorithm}
	\label{alg:algorithm}
	\begin{algorithmic}[1]
		
		\STATE \textbf{Initialization}: Each node $i \in \mathcal{V}$ arbitrary selects  $u_i^{0}\in \mathbb{R}^{m+p}$ and $y_i^{0}=[(x_i^0)^T, (t_i^0)^T]^T \in Y_i$.

		\STATE  Each node $i \!\in \!\mathcal{V}$ sends $x_i^{0}$ and $A_{ij}$ to every neighbor $j $ and sets $q_i^{0}= \max\{t_i^0-g_i(x_{ \mathcal{N}_{i} }^0), \mathbf{0}_{p}\}$. 
		\STATE Each node $i \!\in \!\mathcal{V}$ computes $\nabla_{x_j}f_i(x^{0}_{\mathcal{N}_{i} })+(\frac{\partial g_i (x^{0}_{\mathcal{N}_{i} })}{\partial x_j})^T(q_i^0+g_i(x_{\mathcal{N}_{i} }^{0})\!-t_i^0 )$, sends it with $u_i^0$ to every neighbor $j $, and sets $z_i^{0}\!=\rho \sum_{j \in \mathcal{N}_{i} } [P^H]_{i j} u_{j}^{0} $.
		\FOR {$k=0,1,2,\ldots$, each node  $i \in \mathcal{V}$}
		
		\STATE Compute $d_{i}^{k}$ according to \eqref{eq:jacobianofyi}--\eqref{eq:expressionofdi2}; 
		
		\STATE Update $y_i^{k+1}$ according to \eqref{eq:updateofyi};
	\STATE   Send  $x_i^{k+1}$ to every neighbor $j $;
		\STATE  Update
		$q_{i}^{k+1}$ according to \eqref{eq:updateofQi}; 
		\STATE Update
		$u_{i}^{k+1}$ according to \eqref{eq:updateofui};
		\STATE  Compute $\nabla_{x_j}f_i(x^{k+1}_{\mathcal{N}_{i} })+(\frac{\partial g_i (x^{k+1}_{\mathcal{N}_{i} })}{\partial x_j})^T(q_i^{k+1}+g_i(x_{\mathcal{N}_{i} }^{k+1})\!-t_i^{k+1} )$ and sent it with  $u_i^{k+1} $ to every neighbor $j$;
		\STATE Update
		$z_{i}^{k+1}$ according to \eqref{eq:updateofzi}.
		\ENDFOR
	\end{algorithmic}
\end{algorithm}}
\textbf{Special case of no variable coupling:}
It is noteworthy that at each iteration, each node $i \in \mathcal{V}$ needs to communicate with its neighbors twice. Nevertheless, if each $f_i$, $g_i$ only depends on $x_i$ and each $A_i$ only involves $x_i$, i.e., no variable coupling exists, the calculation of $d_i^k$ in \eqref{eq:jacobianofyi} requires no information except $u_j^k$ from each neighbor $j$, which greatly simplifies the calculation of $d_i^k$. Moreover, in this case, updating $q_i^{k+1}$ in Line 8 of Algorithm~\ref{alg:algorithm} does not rely on $x_j^k$ received from neighbor $j$, and updating $u_i^{k+1}$ in Line 9 involves only $A_i$ and does not need $A_{ji}$. The above simplified operations require each node $i \in \mathcal{V}$ to send only $u_i^k$ to its every neighbor $j$, and the transmission of $x_i^0, A_{ij}$ in Line 2, $x_i^{k+1}$ in Line 7, and  $\nabla_{x_j}f_i(x^{k}_{\mathcal{N}_{i} })+(\frac{\partial g_i (x^{k}_{\mathcal{N}_{i} })}{\partial x_j})^T(q_i^{k}+g_i(x_{\mathcal{N}_{i} }^{k})\!-t_i^{k}) \forall k \ge 0$  in Line 3 and Line 10 of Algorithm~\ref{alg:algorithm} can be removed. Consequently, the communication and computational costs of Algorithm~\ref{alg:algorithm} can be significantly reduced. Additionally, in this case, all the nodes do not have to share their primal decisions, the gradient of local cost functions, the values of local constraint functions, and their Jacobian matrices, thereby reducing the exposure of critical local information.

\begin{remark}
Here, we discuss the differences and connections between Algorithm~\ref{alg:algorithm} and the existing methods applicable to various variable coupling \cite{todescato2020partition,pu2016quantization,HuJH2018,XiaoYY2017,wang2024distributed,Alghunaim2019,Mota2014}.
		By letting each node maintain local variables and the copies of its neighbors' variables, \cite{HuJH2018} transforms its locally-coupled problem into an equality-constrained problem and proposes several algorithms based on operator splitting techniques to solve this equivalent problem. A similar strategy is adopted in \cite{XiaoYY2017} to address convex feasibility problems. The methods in \cite{Mota2014, Alghunaim2019} are able to tackle local functions involving non-neighbors' variables (yet they cannot address globally-coupled constraints). In the special case where the local functions only depend on neighboring variables, they propose different cluster partition approaches, leading to the extended ADMM in \cite{Mota2014} and the coupled diffusion strategy in \cite{Alghunaim2019}, where each node also needs to update its local variable and the copies of its neighbors' variables. Clearly, the methods in \cite{HuJH2018, XiaoYY2017, Alghunaim2019, Mota2014}, incur higher memory and communication costs compared to Algorithm~\ref{alg:algorithm}, especially for densely connected networks. 
		
	Like Algorithm~\ref{alg:algorithm}, the methods in \cite{todescato2020partition,pu2016quantization,wang2024distributed} allow each node to maintain its own decision variables only. Notably, in the absence of constraints or with only local set constraints, Algorithm~\ref{alg:algorithm} is equivalent to the synchronous gradient-based method in \cite{todescato2020partition} and the unquantized version of the algorithm in \cite{pu2016quantization}. Furthermore, by setting $W=I$, $\mathbf{z}^k=\mathbf{0}$, and $\mathbf{y}^k=\mathbf{x}^k$, our algorithm reduces to the method in \cite{wang2024distributed} which only considers variable coupling under locally-coupled constraints. In contrast with \cite{todescato2020partition,pu2016quantization,HuJH2018,XiaoYY2017,wang2024distributed,Alghunaim2019,Mota2014}, our algorithm is able to simultaneously handle globally-coupled nonlinear inequality and linear equality constraints, making it applicable to a broader class of problems.
\end{remark}

\section{Convergence Analysis}\label{sec:convanal}

This section is dedicated to analyzing the convergence performance of Algorithm~\ref{alg:algorithm}, along with a comparison to related works and a discussion on the parameter selections for guaranteeing the theoretical results.

\subsection{Convergence Results} 
To establish the convergence analysis, we define $\mathbf{z}^\star=\mathbf{B}\mathbf{y}^{\star}\!-\mathbf{c}$ and let $(\bm{\lambda}^\star, u^\star)$ be an optimal solution of the Lagrange dual problem of problem~\eqref{eq:probleminy}, where $\bm{\lambda}^\star \ge \mathbf{0}_{np}$. Also, let $\mathbf{u}^\star=\mathbf{1}_n\otimes u^\star$. We keep track of the running average $\bar{\mathbf{y}}^k:=\frac{1}{k}\sum\limits_{\ell=1}^k \mathbf{y}^{\ell}$ $\forall k \ge 1$ to establish the convergence analysis. The following two theorems state the convergence results of Algorithm~\ref{alg:algorithm}.

\begin{theorem} \label{theo:feasibility}
Suppose Assumptions~\ref{asm:primalprob}--\ref{asm:matricesphpw} hold. Let $\gamma, \rho>0$ be such that $C:=\frac{1}{2\rho}\|\mathbf{z}^{0}-\mathbf{z}^\star\|_{H^\dag}^2+\frac{\rho}{2}(\|\mathbf{u}^{0}\|_W+\|\mathbf{u}^\star\|_W)^2+\frac{1}{2}\|\mathbf{G}(\mathbf{y}^0)\|^2+\frac{1}{2}\|\mathbf{G}(\mathbf{y}^\star)\|^2+\frac{1}{2}\|\mathbf{q}^{0}\|^2+\|\bm{\lambda}^\star\|^2+\frac{1}{2}(\mathbf{y}^{0}-\mathbf{y}^\star)^T(\frac{I}{\gamma}-\frac{\mathbf{B^T}\mathbf{B}}{\rho})(\mathbf{y}^{0}-\mathbf{y}^\star)\ge0$. Also suppose
\begin{align}
	\frac{1}{\gamma} \ge\frac{1}{\rho}\|\mathbf{B}^{T}\mathbf{B}\|+\tilde{\beta}^2+L_F+4\sqrt{np}(\|\bm{\lambda}^\star\|+\sqrt{C})L_g,\label{eq:parametergamma}
\end{align}
where $L_F$, $\tilde{\beta}$, and $L_g$ are given in Proposition~\ref{pro:propertyofproby}(\ref{pro:gardfsmooth})(\ref{pro:Gycontinuous}) and Assumption~\ref{asm:primalprob}(\ref{gijsmooth}), respectively. Then,
\begin{align}
	&\mathbf{G}(\bar{\mathbf{y}}^k)\le \frac{2(\|\bm{\lambda}^\star\|+\sqrt{C})}{k}\mathbf{1}_{np},\quad\forall k\ge 1,\label{eq:Gybarkupperbound}\displaybreak[0]\\
	&\|(\mathbf{1}_{n}\otimes I_{m+p})^T(\mathbf{B}\bar{\mathbf{y}}^k-\mathbf{c})\|\le \frac{D^0}{k},\quad\forall k\ge1.\label{eq:Bybarcupperbound}
\end{align}
where $D^0=\rho\sqrt{n}(\|\mathbf{u}^{0}\|_W+\|\mathbf{u}^\star\|_W+\sqrt{2C/\rho})$.
\end{theorem}

\begin{proof}
See Appendix~\ref{ssec:proofoftheorem1}.
\end{proof}

The inequality \eqref{eq:parametergamma} specifies the conditions on the step-size $\gamma$ and the parameter $\rho$ required for Algorithm~1 to achieve the convergence rates in Theorem~\ref{theo:feasibility} and Theorem~\ref{theo:funcval}. The existence of positive $\gamma$ and~$\rho$ satisfying~\eqref{eq:parametergamma} will be justified in Section~\ref{ssec:parameter}. Additionally, although those theoretical parameter conditions involve the unavailable primal and dual optima, Section~\ref{ssec:parameter} demonstrates that they can be satisfied based on accessible information only.

\begin{theorem} \label{theo:funcval}
Under all the conditions in Theorem~\ref{theo:feasibility},
\begin{align}
	-\frac{C^0}{k}\le \mathbf{f}(\bar{\mathbf{y}}^k)\!-\!\mathbf{f}(\mathbf{y}^\star)\le\frac{S^{0}}{k},\quad\forall k\ge 1,\label{eq:theofuncvalupperbound} 
\end{align}
where
\begin{align*}
	C^0=&2(\|\bm{\lambda}^\star\|+\sqrt{C})\mathbf{1}_{np}^T\bm{\lambda}^\star\displaybreak[0]\\
	&+\rho\|\mathbf{u}^\star\|(\|\mathbf{u}^{0}\|_W+\|\mathbf{u}^\star\|_W+\sqrt{2C/\rho}),\displaybreak[0]\\
	S^0=&\frac{1}{2\rho}\|\mathbf{z}^{0}-\mathbf{z}^\star\|_{H^\dag}^2+\frac{\rho}{2}\|\mathbf{u}^{0}\|_W^2+\frac{1}{2}\|\mathbf{y}^{0}-\mathbf{y}^\star\|^2_{\frac{I}{\gamma}-\frac{\mathbf{B^T}\mathbf{B}}{\rho}}\displaybreak[0]\\
	&+\frac{1}{2}\|\mathbf{q}^{0}\|^2-\frac{1}{2}\|\mathbf{G}(\mathbf{y}^{0})\|^2.
\end{align*}	
\end{theorem}
\begin{proof}
See Appendix \ref{proofoftheorem2}.
\end{proof}

Theorem~\ref{theo:feasibility} and Theorem~\ref{theo:funcval} show that the proposed algorithm achieves $O(1/k)$ convergence rates for both constraint violations and objective function value to problem~\eqref{eq:probleminy}. Due to the equivalence between problem~\eqref{eq:primalprob} and problem~\eqref{eq:probleminy}, the above results also indicate $O(1/k)$ convergence of constraint violations and objective function value to problem~\eqref{eq:primalprob}, which is described in Corollary~\ref{theo:originalprobconv}. In Corollary~\ref{theo:originalprobconv}, $\bar{x}_{\mathcal{N}_{i}}^{k}$ $\forall k\ge1$ is the vector obtained by stacking $\bar{x}_j^k:=\frac{1}{k}\sum_{\ell=1}^kx_j^\ell$ $\forall j\in\mathcal{N}_{i}$.

\begin{corollary}\label{theo:originalprobconv}
Suppose all the conditions in Theorem~\ref{theo:feasibility} hold. For each $k\ge1$,
\begin{align}
	&\sum_{i\in \mathcal{V}} g_i(\bar{x}_{\mathcal{N}_{i}}^{k})\le \frac{2n(\|\bm{\lambda}^\star\|+\sqrt{C})+D^0}{k}\mathbf{1}_p,\label{eq:globineqconv}\displaybreak[0]\\
	&\|{\sum_{i \in \mathcal{V}}} (A_i\bar{x}_{\mathcal{N}_{i}}^k-b_i)\| \le\frac{D^0}{k},\label{eq:priequalitybound}\displaybreak[0]\\
	& -\frac{C^0}{k} \le\Bigl(\sum_{i\in \mathcal{V}} f_i(\bar{x}_{\mathcal{N}_{i}}^k)\Bigr)-f^\star\le\frac{S^{0}}{k},\label{eq:funcvalupperbound}
\end{align}
where $C$, $D^0$, $C^0$, and $S^0$ are defined in Theorems~\ref{theo:feasibility} and~\ref{theo:funcval}, and $f^\star$ is the optimal value of problem~\eqref{eq:primalprob}.
\end{corollary}

\begin{proof}
See Appendix \ref{proofoftheorem3}.
\end{proof}

\begin{remark}\label{rmk:networkconnectivity}
The network connectivity influences the convergence of constraint violations and objective error via the smallest nonzero eigenvalue of $H$, i.e., $\lambda_{\min}(H)$. 
This is because $\|\mathbf{z}^{0}-\mathbf{z}^\star\|_{H^\dag}^2$, which determines the bounds on the right-hand side of  \eqref{eq:globineqconv}--\eqref{eq:funcvalupperbound}, is bounded by $\frac{1}{\lambda_{\min}(H)}\|\mathbf{v}\|^2$, where  $\mathbf{v}\in \mathbb{R}^{n(m+p)}$ such that $\mathbf{z}^0-\mathbf{z}^\star =H\mathbf{v}$. Therefore, larger $\lambda_{\min}(H)$, often indicating denser network connectivity, leads to faster convergence. On the other hand, denser network connectivity increases the computation and communication costs during each iteration. For example, adding a neighbor $\ell$ to node $i$ requires computing additional $d_\ell$ variables and requires transmitting $d_\ell + d_i + m + p$ additional real numbers.
\end{remark}

\begin{table*}[t]
	\centering
	\renewcommand*{\thetable}{1}
	\footnotesize
	\caption{ Comparison with related methods in solving problems with globally-coupled inequality constraints yet with no variable coupling in local functions.\label{table:comparison}}
	\resizebox{1\textwidth}{!}{ 	
		\begin{tabular}{|p{6.22cm}|p{0.85cm}|p{0.95cm}|p{1.22cm}|p{2cm}|p{1cm}|p{1.5cm}|}
			\hline
			\hfil Algorithm  & compact domain & smooth problem& existence of equality constraint & \centering{main computational cost} & constant step-size & convergence result\\
			\hline
			\hfil consensus-based primal-dual perturbation method 
			\cite{THChang2014} & \hfil $\surd$ &\hfil $\surd$ & &\hfil projection &&\hfil asymptotic\\
			\hline
			\hfil dual decomposition based method \cite{Falsone2017} &  \hfil$\surd$ & &&  \hfil subproblem &&\hfil asymptotic \\
			\hline
			\hfil RSDD \cite{Notarnicola2019} & \hfil$\surd$&  && \hfil subproblem&&\hfil asymptotic\\	
			\hline
			\hfil primal-dual gradient method 
			\cite{LiangS2019} & \hfil$\surd$ & \hfil$\surd$ &\hfil $\surd$& \hfil projection &\hfil$\surd$   &\hfil asymptotic\\	
			\hline
			\hfil	ALT \cite{falsone2023augmented}
			&&&\hfil$\surd$&\hfil subproblem &\hfil $\surd$&\hfil asymptotic\\	
			\hline
			\hfil	dual subgradient method \cite{LiangS2019a} &  \hfil $\surd$& &\hfil $\surd$ &\hfil subproblem && $O(\ln k/\sqrt{k})$\\
			\hline
			\hfil DPD-TV\cite{camisa2021distributed}&\hfil$\surd$&&&\hfil  subproblem &&\hfil $O(1/\ln k)$\\
			\hline
			\hfil $\text{DSA}_2$ \cite{LiuC2020}
			& \hfil$\surd$ & &&\hfil subproblem & &\hfil $O(1/\sqrt{k})$\\
			\hline
			\hfil	C-SP-SG \cite{Mateos2016}&\hfil$\surd$&&&\hfil projection &&\hfil $O(1/\sqrt{k})$\\
			\hline 
			\hfil	B-DPP \cite{wang2023distributed}&\hfil$\surd$&&&\hfil subproblem&&\hfil $O(1/\sqrt{k})$\\
			\hline
			\hfil	IPLUX \cite{wu2022distributed}&&&\hfil$\surd$&\hfil  subproblem&\hfil$\surd$&\hfil $O(1/k)$\\
			\hline
			\hfil	Algorithm~\ref{alg:algorithm} &\hfil$\surd$ &\hfil $\surd$ &\hfil$\surd$   &\hfil projection  &\hfil $\surd$&\hfil $O(1/k)$\\
			\hline
	\end{tabular}}
\end{table*}

\subsection{Comparative Discussion}\label{ssec:compdiss}
To the best of our knowledge, no prior methods can be directly applied to solve problem~\eqref{eq:primalprob} in a decentralized way. Nevertheless, the algorithms in \cite{THChang2014,Falsone2017,Notarnicola2019,LiangS2019,LiangS2019a,LiuC2020,camisa2021distributed ,falsone2023augmented,Mateos2016,wang2023distributed,wu2022distributed} are able to solve problems that admit globally-coupled inequality constraints yet \textit{no variable coupling}, where each node's local objective and constraint functions only depend on its own decision variable and are independent of its neighbors' decisions. Here, we compare Algorithm~\ref{alg:algorithm} with those algorithms for solving such \textit{specialized forms} of \eqref{eq:primalprob}.

Observe from TABLE~\ref{table:comparison} that the algorithms in \cite{THChang2014,Falsone2017,Notarnicola2019,LiangS2019,falsone2023augmented} do not have guaranteed convergence rates. The convergence rate results in \cite{LiangS2019a,LiuC2020,Mateos2016,camisa2021distributed,wang2023distributed} are weaker than the $O(1/k)$ rate for Algorithm~\ref{alg:algorithm}, while they allow certain nonsmoothness in their problems. Although IPLUX \cite{wu2022distributed} achieves $O(1/k)$ convergence rates in no need of a compact domain, it needs every node to solve a local convex optimization subproblem at each iteration (same for the algorithms in \cite{wang2023distributed,LiangS2019a,falsone2023augmented,camisa2021distributed,Notarnicola2019,Falsone2017,LiuC2020}), so that it has higher computational complexity than Algorithm 1, which adopts a projection operation, i.e., minimizing a positive definite quadratic function subject to the same convex constraints. Notably, only Algorithm~\ref{alg:algorithm} and \cite{LiangS2019,LiangS2019a, wu2022distributed,falsone2023augmented} consider problems with both inequality and equality constraints, while the others do not include equality constraints in their problem formulations. Although we may transform each equality constraint into two inequalities with reverse signs, this does not work for \cite{THChang2014,Mateos2016,camisa2021distributed} because their methods require each inequality constraint to be strictly negative at a Slater point in order to bound the dual optimal set, and would significantly increase the variable dimensions and computational/communication complexities of the algorithms in \cite{Falsone2017,Notarnicola2019,LiuC2020,wang2023distributed}. Furthermore, Algorithm~\ref{alg:algorithm} and \cite{falsone2023augmented,LiangS2019, wu2022distributed} adopt constant step-sizes, whereas the remaining algorithms in TABLE~\ref{table:comparison} use diminishing step-sizes, which may lead to slow convergence in practice.

Furthermore, we would like to point out that the methods in [15]–[25] cannot be extended to tackle problem~\eqref{eq:primalprob} with variable coupling in a straightforward way. This is because the distributed methods in \cite{LiangS2019a,falsone2023augmented,camisa2021distributed,Notarnicola2019,Falsone2017,LiuC2020} inherently rely on the separability of the Lagrangian dual problem, which, however, cannot be decomposed into $n$ local optimization problems with $x_i$ being the sole optimization variable when the variable coupling in $f_i$ and $g_i$ exists. The works \cite{wang2023distributed, wu2022distributed} update local primal variables by solving convex optimization problems that involve $f_i$ and $g_i$ at each iteration, which also cannot be implemented in a decentralized manner due to the variable coupling structure. The projection-based methods \cite{THChang2014,LiangS2019, Mateos2016} either fail to handle the equality constraints \cite{THChang2014, Mateos2016} as we stated before, or cannot ensure convergence \cite{LiangS2019} when addressing problem~\eqref{eq:primalprob}.

\subsection{Parameter Selections}\label{ssec:parameter}
In this subsection, we show that there exist $\gamma,\rho>0$ satisfying the parameter conditions in Theorem~\ref{theo:feasibility}, i.e., $C\ge0$ and \eqref{eq:parametergamma}, and then discuss how to select the parameters under those theoretical conditions in practice.

For simplicity, we arbitrarily fix $\rho>0$. To achieve $C\ge0$, we may simply let $0<\gamma\le\frac{\rho}{\|\mathbf{B}^{T}\mathbf{B}\|}$. When all the inequality constraint functions $g_i(x_{\mathcal{N}_i })$ $\forall i \in \mathcal{V}$ are affine, \eqref{eq:parametergamma} reduces to $0<\gamma\le(\tilde{\beta}^2+L_F +\frac{1}{\rho}\|\mathbf{B}^{T}\mathbf{B}\|)^{-1}$ because $L_g=0$. When $g_i(x_{\mathcal{N}_i })$ $\forall i \in \mathcal{V}$ are nonlinear, below we provide a range for $\gamma$ that sufficiently guarantees \eqref{eq:parametergamma}.

To do so, define $D:=\tilde{\beta}^2+ L_F+4\sqrt{np}L_g\|\bm{\lambda}^\star\|+\frac{1}{\rho}\|\mathbf{B}^{T}\mathbf{B}\|$ and $R:=C+\frac{1}{2\rho}\|\mathbf{y}^{0}-\mathbf{y}^\star\|_{\mathbf{B^T}\mathbf{B}}^2-\frac{1}{2\gamma}\|\mathbf{y}^{0}-\mathbf{y}^\star\|^2\ge0$ for convenience. Note from the definition of $C$ in Theorem~\ref{theo:feasibility} that $D$ and $R$ are both independent of $\gamma$. Then, \eqref{eq:parametergamma} can be rewritten as 
\begin{align}
\gamma D+ 4\gamma\sqrt{np}L_g\sqrt{C}\le1.\label{eq:gamma(D+sqrtC)le0}
\end{align}
Next, we find an upper bound on the left-hand side of \eqref{eq:gamma(D+sqrtC)le0}. Note that $\gamma D+ 4\gamma\sqrt{np}\sqrt{C}L_g\le \gamma D+ 4\sqrt{np}L_g\sqrt{\gamma^2R+\frac{\gamma}{2}\|\mathbf{y}^{0}-\mathbf{y}^\star\|^2}\le \gamma(D+ 4\sqrt{np}L_g\sqrt{R})+4\sqrt{\gamma}\sqrt{np}L_g\|\mathbf{y}^{0}-\mathbf{y}^\star\|$. Clearly, setting this upper bound no more than $1$ suffices to guarantee \eqref{eq:gamma(D+sqrtC)le0} and equivalently \eqref{eq:parametergamma}, which can be achieved by letting 
\begin{align}
0<\gamma\le\tilde{\gamma},\label{eq:gammarange} 
\end{align}
where  $\tilde{\gamma}:=\Bigl(\sqrt{4L_g^2np\|\mathbf{y}^{0}-\mathbf{y}^\star\|^2+(D+4\sqrt{np}L_g\sqrt{R})}+2L_g\sqrt{np}\|\mathbf{y}^{0}-\mathbf{y}^\star\|\Bigr)^{-2}$ is independent of $\gamma$ and $\tilde{\gamma}\le \frac{\rho}{\|\mathbf{B}^{T}\mathbf{B}\|}$. Therefore, for any given $\rho >0$, there always exists a sufficiently small $\gamma\in(0,\tilde{\gamma}]$ such that $C\ge0$ and \eqref{eq:parametergamma} holds. 

Such a sufficiently small $\gamma$ can be obtained either empirically (e.g., by trial and error) or analytically (e.g., by estimating $\tilde{\gamma}$). To do the latter, we may construct some lower bound on $\tilde{\gamma}$ without the knowledge of the optima $\mathbf{y}^\star, \mathbf{u}^\star, \bm{\lambda}^\star, \mathbf{z}^\star$, leading to a subinterval of $(0,\tilde{\gamma}]$ that still ensures $C\ge0$ and \eqref{eq:parametergamma}. The main challenge here is to bound $\|\mathbf{y}^\star\|$, $\|\mathbf{u}^\star\|_W$, $\|\bm{\lambda}^\star\|$, $\|\mathbf{z}^\star\|$ by means of available information only, and some feasible approaches are provided as follows. 
\begin{enumerate}[(1)]
\item \emph{Upper bound of $\|\mathbf{y}^\star\|$:} Recall that $\mathbf{y}^\star$ is composed of $\mathbf{x}^\star$ and $t^\star$. We may utilize the diameter of the compact set $X_{\mathcal{V}}$ and any point in $X_{\mathcal{V}}$ to bound $\|\mathbf{x}^\star\|$. Additionally, we may choose $t_i^\star=g_i(x_{\mathcal{N}_{i}}^\star)-\frac{1}{n}\sum_{j\in \mathcal{V}} g_j(x_{\mathcal{N}_{j}}^\star)$ $\forall i\in \mathcal{V}$ (see proof of Proposition~1(d) for its optimality and feasibility), and its upper bound can be acquired from the continuity of each $g_j$ over the compact $X_{\mathcal{N}_j}$. 

\item \emph{Upper bound of $\|\mathbf{z}^\star\|$:} Since $\mathbf{z}^\star=\mathbf{B}\mathbf{y}^{\star}\!-\mathbf{c}$, its upper bound can be directly obtained using the bound on $\|\mathbf{y}^\star\|$.

\item \emph{Upper bound of $\|\bm{\lambda}^\star\|$:} By virtue of the Slater's condition of problem~\eqref{eq:probleminy} and the continuity of $\sum_{i\in \mathcal{V}} f_i(x_{\mathcal{N}_{i} })$ over the compact $X_{\mathcal{V}}$, the method in \cite{nedic2009approximate} can be employed to derive an upper bound on $\|\bm{\lambda}^\star\|$.

\item \emph{Upper bound of $\|\mathbf{u}^\star\|_W$:} Note that the Lagrangian function of \eqref{eq:probleminy} can be expressed as $L(\mathbf{y},\bm{\lambda}, \mathbf{u})= \mathbf{f}(\mathbf{y})+\langle \bm{\lambda},\mathbf{G}(\mathbf{y})\rangle+\langle \mathbf{u}, \mathbf{B}\mathbf{y}-\mathbf{c}\rangle$, where $\mathbf{u}=[u_1^T,\ldots,u_n^T]^T$ with $u_1 = \cdots = u_n$. Let the dual optima be written as $\bm{\lambda}^\star=[(\bm{\lambda}_1^\star)^T, \cdots, (\bm{\lambda}_n^\star)^T]^T\in \mathbb{R}^{np}$ and $u_{i}^\star=[(u_{i1}^\star)^T, (u_{i2}^\star)^T]^T$ $\forall i\in \mathcal{V}$ with $u_{i1}^\star\in \mathbb{R}^{m}$ and $u_{i2}^\star\in \mathbb{R}^{p}$. From the stationarity condition in the Karush-Kuhn-Tucker Conditions, we obtain $u_{i2}^\star=\bm{\lambda}_i^\star$ $\forall i\in \mathcal{V}$ and
\begin{align} 
	\nabla_\mathbf{x} \mathbf{f}(\mathbf{y}^\star) + \nabla_\mathbf{x} \mathbf{G}(\mathbf{y}^\star)\bm{\lambda}^\star+\bar{A}\bar{\mathbf{u}}_1^\star=\mathbf{0}, \label{eq:findui1}
\end{align}
where $\bar{A}=\operatorname{diag}(\bar{A}_1, \cdots, \bar{A}_n)$ with each $\bar{A}_i=(\sum_{j \in \mathcal{N}_i}A_{ji})^T$ and $\bar{\mathbf{u}}_1^\star=[(u_{11}^\star)^T, \cdots, (u_{n1}^\star)^T]^T$. Since $W\mathbf{u}^\star=\mathbf{u}^\star$, we have $\|\mathbf{u}^\star\|^2_W =\|\mathbf{u}^\star\|^2=\sum_{i=1}^n(\|u_{i1}^\star\|^2+\|u_{i2}^\star\|^2)=\|\bar{\mathbf{u}}_1^\star\|^2+\|\bm{\lambda}^\star\|^2$. Hence, it remains to bound $\|\bar{\mathbf{u}}_1^\star\|$. This may be realized by viewing \eqref{eq:findui1} as a linear equation with $\bar{\mathbf{u}}_1^\star$ being the unknown, which is guaranteed to have a solution, and deriving the expression of any solution. For example, when each $\bar{A}_{i}$ has full column rank, which occurs in many real-world problems such as task offloading over mobile edge computing networks \cite{xiao2018distributed} and economic dispatch for power systems \cite{yang2016distributed}, the linear equation \eqref{eq:findui1} has a unique solution 
\begin{align*}
	\bar{\mathbf{u}}_1^\star=-(\bar{A}^T\bar{A})^{-1}\bar{A}^T(\nabla_\mathbf{x} \mathbf{f}(\mathbf{y}^\star)+ \nabla_\mathbf{x} \mathbf{G}(\mathbf{y}^\star)\bm{\lambda}^\star). 
\end{align*}
When each $\bar{A}_{i}$ has full row rank, the minimum-norm solution is given by 
\begin{align*}
	\bar{\mathbf{u}}_1^\star=-\bar{A}^T(\bar{A}\bar{A}^T)^{-1}\!(\nabla_\mathbf{x} \mathbf{f}(\mathbf{y}^\star) +\nabla_\mathbf{x} \mathbf{G}(\mathbf{y}^\star)\bm{\lambda}^\star).
\end{align*}  
Thus, using the aforementioned bounds on $\mathbf{y}^\star,\bm{\lambda}^\star$ and the smoothness of $\mathbf{f},\mathbf{G}$, we can find an upper bound on $\|\bar{\mathbf{u}}_1^\star\|$ and therefore an upper bound on $\|\mathbf{u}^\star\|^2_W =\|\mathbf{u}^\star\|^2$.
\end{enumerate}

Upon bounding the above primal and dual optima, we are able to estimate the remaining quantities in $\tilde{\gamma}$, which involve certain problem and graph characteristics as well as the initial states, through mild global coordination or proper distributed computation methods. Such additional costs are minor in comparison with solving the complicated optimization problem~\eqref{eq:primalprob}.

\section{Numerical Experiments}\label{sec:example}
In this section, we validate our theoretical results by applying Algorithm~\ref{alg:algorithm} to two numerical examples. We first consider a globally-coupled constrained problem with no variable coupling, which allows us to compare the proposed algorithm with related works. The second numerical experiment intends to verify the effectiveness of Algorithm 1 on handling both variable coupling and globally-coupled constraints.
\begin{figure}[t]
\centering
\includegraphics[height=5.5cm,width=8.6cm]{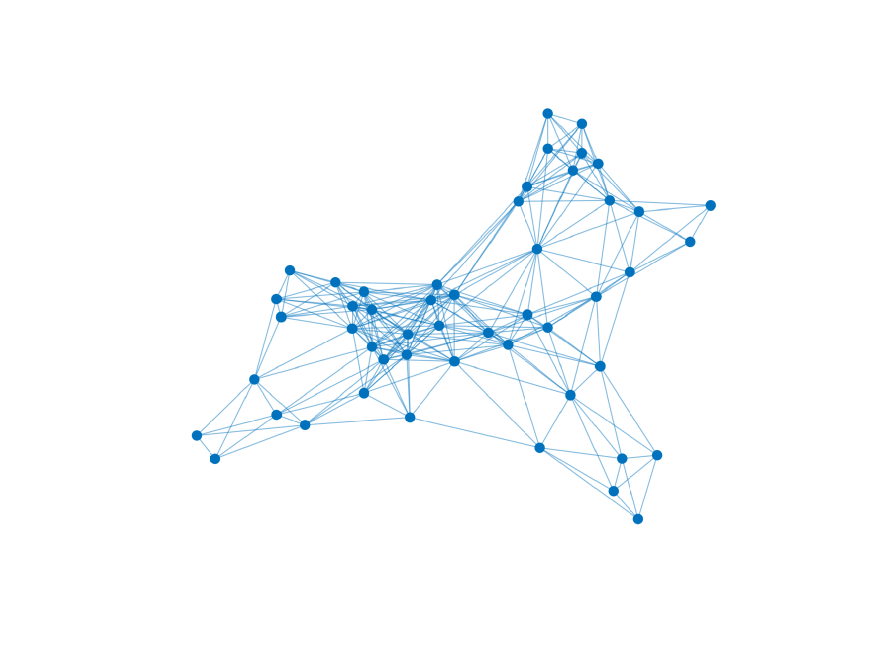}
\caption{Network topology with $n=50$ nodes}
\label{network}
\end{figure}

\subsection{First Numerical Example }\label{subsec:Example1_PEV}
Consider the following problem
\begin{equation}\label{eq:numercialexamp1}
\begin{array}{cl}
	{\underset{x_i \in [0,1] }{\operatorname{minimize}}} ~& {\sum\limits_{i \in \mathcal{V}}} c_ix_i \\
	{\operatorname { subject~to }} & {\sum\limits_{i \in \mathcal{V}}} -d_i\operatorname{log}(1+x_i) \leq -b.
\end{array}	
\end{equation}
\begin{figure}
	\centering
	\subfigure[]{\includegraphics[height=6cm,width=8cm]{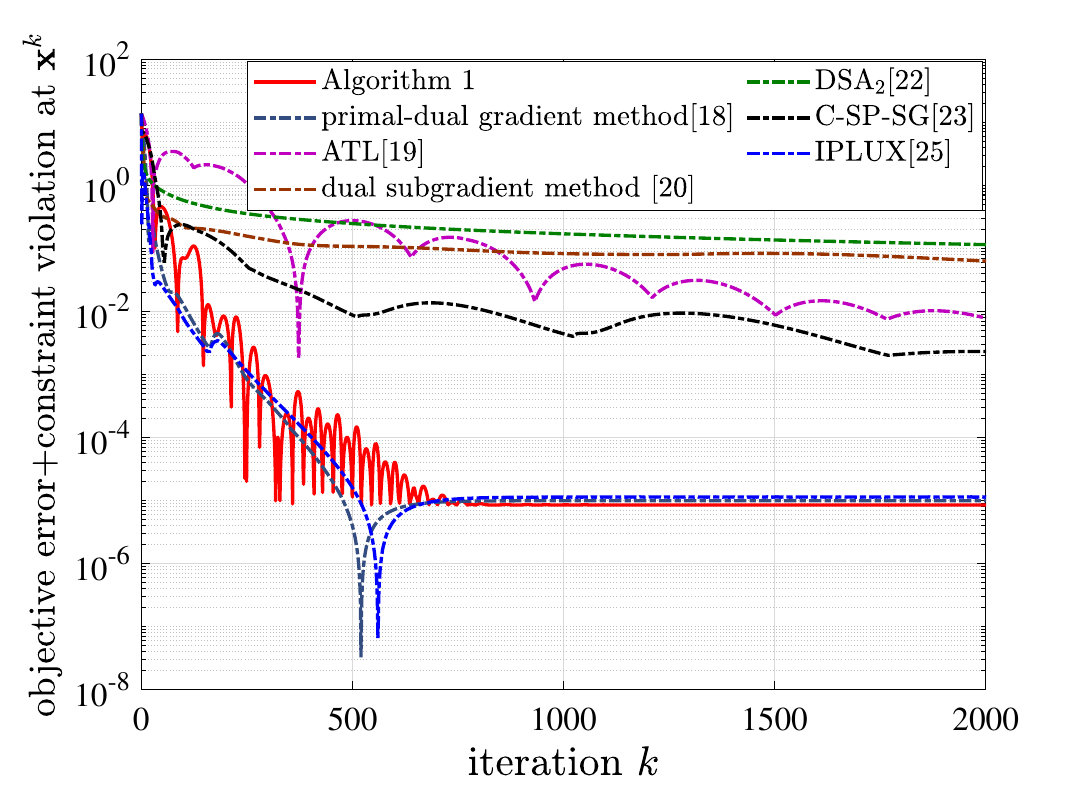}}
	\subfigure[]{	\includegraphics[height=6cm,width=8cm]{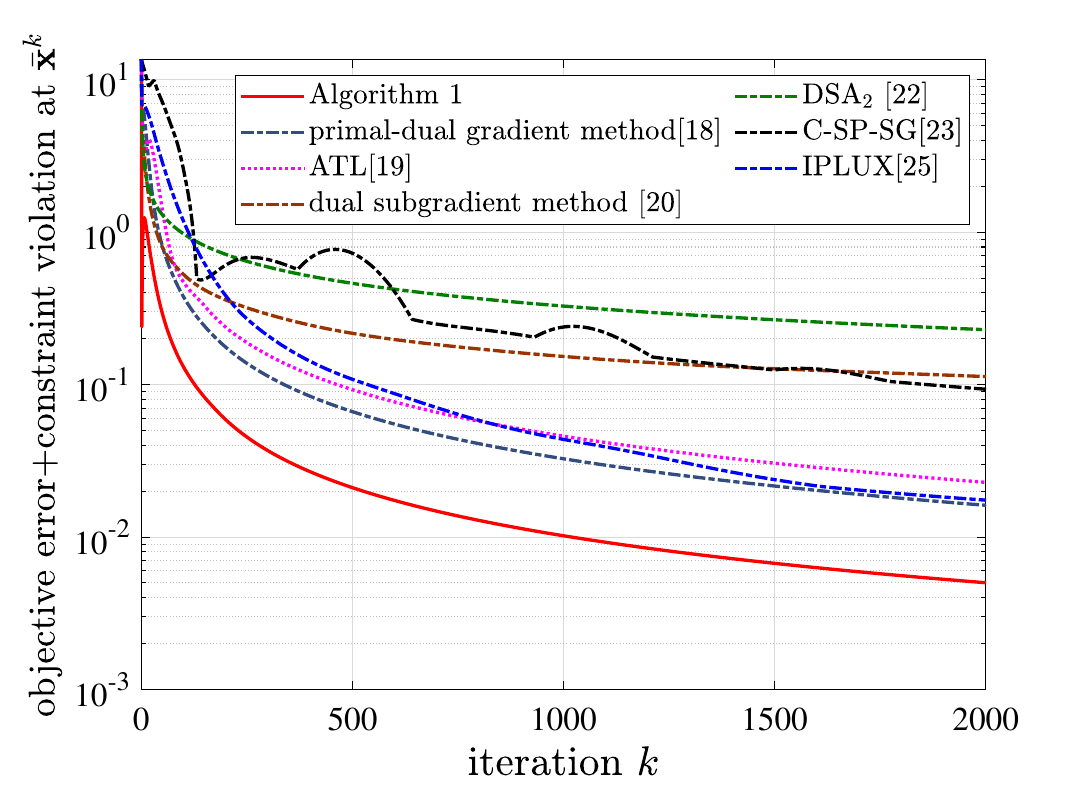}}
	\caption{Convergence performance of Algorithm~\ref{alg:algorithm} and alternative methods for solving problem~\eqref{eq:numercialexamp1}.}
	\label{ex1:optimal_error}
\end{figure}
Problem~\eqref{eq:numercialexamp1} arises in, for example, the quality
service of wireless networks and the plug-in electric vehicles charging problem \cite{LiuC2020,Mateos2016}. In this simulation, we consider a undirected network with $n=50$ nodes (see Fig.~\ref{network}), and the set of nodes is $\mathcal{V}=\{1, \ldots,50\}$. For each $i \in \mathcal{V}$, the constants $c_i, d_i$ are chosen randomly from a uniform distribution in $[0, 1]$ and $b = 5$. Clearly, the linear objective and smooth constraints in problem~\eqref{eq:numercialexamp1} meet Assumption~\ref{asm:primalprob}\eqref{compactXi}--\eqref{giLipschitz}. The Slater's condition can be satisfied when $x_i=1~\forall i \in \mathcal{V}$. Problem~\eqref{eq:numercialexamp1} satisfies
not only Assumption~\ref{asm:primalprob} but also the assumptions in \cite{LiangS2019,LiangS2019a,Mateos2016,falsone2023augmented,LiuC2020,wu2022distributed}.

Under the above settings, we compare Algorithm~\ref{alg:algorithm} with the algorithms proposed in \cite{LiangS2019,LiangS2019a,Mateos2016,falsone2023augmented,LiuC2020,wu2022distributed}. For a fair comparison, we hand-tune all the algorithm parameters to  achieve the best possible convergence performance. The optimal solution of problem~\eqref{eq:numercialexamp1} $\mathbf{x}^\star=[(x_1^\star)^T, \ldots, (x_n^\star)^T]^T$ is obtained by using the solver \textit{fmincon} in Matlab. Fig.~\ref{ex1:optimal_error} plots the sum of the objective error and constraint violation at $\mathbf{x}^k$ and $\bar{\mathbf{x}}^k$, respectively, i.e., $| \sum_{i=1}^{50}c_ix_i^k- c_ix_i^\star|+\|\max\{\sum_{i=1}^{50}-d_i\operatorname{log}(1+x_i^k)+b, 0\}\|$ and $| \sum_{i=1}^{50}c_i\bar{x}_i^k\!-\! c_ix_i^\star|+\|\max\{\sum_{i=1}^{50}-d_i\operatorname{log}(1+\bar{x}_i^k)+b, 0\}\|$, where $\mathbf{x}^k=[(x_1^k)^T,(x_2^k)^T,\ldots,(x_n^k)^T]^T$ is the network-wide primal iterate at iteration $k$, and $\bar{\mathbf{x}}^k:=\sum\limits_{\ell=1}^k \mathbf{x}^{\ell}/k$ is its running average. 
 
From Fig.~\ref{ex1:optimal_error}, we observe that Algorithm \ref{asm:primalprob} converges faster than the algorithms in \cite{falsone2023augmented,LiangS2019a,LiuC2020,Mateos2016}. It also achieves a comparable or slightly better convergence rate compared to the algorithms in \cite{LiangS2019} and \cite{wu2022distributed}. Nevertheless, \cite{LiangS2019} needs to transmit four times the number of variables as Algorithm~\ref{alg:algorithm} per iteration, leading to much higher communication cost, and \cite{wu2022distributed} requires higher computational costs per iteration than Algorithm~\ref{alg:algorithm}. The numerical results demonstrate the competitive convergence
 performance of our proposed algorithm.

\subsection{Second Numerical Example}
Consider the following problem
\begin{equation} \label{eq:numercialexamp2}
\begin{array}{cl}
	{\underset{x_i \in \mathbb{R}^{d}, \forall i \in \mathcal{V} }{\operatorname{minimize}}} ~& \!\sum\limits_{i \in \mathcal{V}} (x_{\mathcal{N}_{i}})^TP_ix_{\mathcal{N}_{i}}+Q_i^Tx_{\mathcal{N}_{i}} \\
	{\operatorname { subject~to }} & \!\sum\limits_{i \in \mathcal{V}} (x_{\mathcal{N}_{i}})^T\!\!A_ix_{\mathcal{N}_{i}}\!+\!a_i^Tx_{\mathcal{N}_{i}}\!\leq\! \sum\limits_{i \in \mathcal{V}}\!d_i, \\
	{}&\!\sum\limits_{i \in \mathcal{V}} B_ix_{\mathcal{N}_{i}} = {\sum\limits_{i \in \mathcal{V}}}b_i,\\
	{} & \!x_i \in \{x \in \mathbb{R}^d : l_i \le x \le u_i\},
\end{array}
\end{equation}
where each $P_i, A_i \in \mathbb{R}^{d
|\mathcal{N}_i|\times d
|\mathcal{N}_i|}$ is symmetric positive semidefinite for guaranteeing the convexity of local functions. Besides, $Q_i, a_i \in  \mathbb{R}^{d|\mathcal{N}_i|}$, $B_i \in  \mathbb{R}^{m \times d|\mathcal{N}_i|}, b_i \in \mathbb{R}^{m}$, and $d_i \in \mathbb{R}$. These problem data are randomly generated under the above conditions. Under the same network environment as Fig.~\ref{network}, let $m=d=2$ and $X_i=\{x \in \mathbb{R}^d : l_i \le x \le u_i\}$, $f_{i}\left(x_{\mathcal{N}_{i}}\right)=(x_{\mathcal{N}_{i}})^TP_ix_{\mathcal{N}_{i}}+Q_i^Tx_{\mathcal{N}_{i}}$ and $g_{i}\left(x_{\mathcal{N}_{i}}\right)=(x_{\mathcal{N}_{i}})^TA_ix_{\mathcal{N}_{i}}+a_i^Tx_{\mathcal{N}_{i}}-d_i$. Apparently, problem~\eqref{eq:numercialexamp2} satisfies Assumption~\ref{asm:primalprob} under proper problem data.
\begin{figure}
	\centering
	\subfigure[]{\includegraphics[height=6cm,width=8cm]{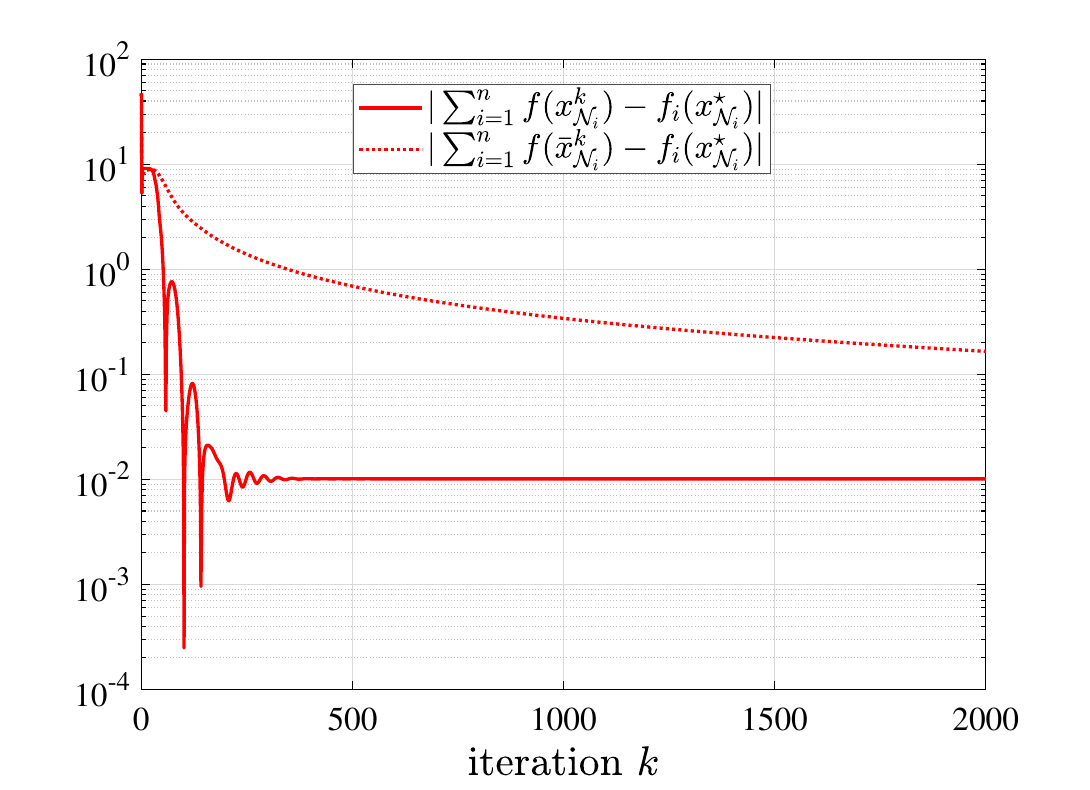}}
	\subfigure[]{	\includegraphics[height=6cm,width=8cm]{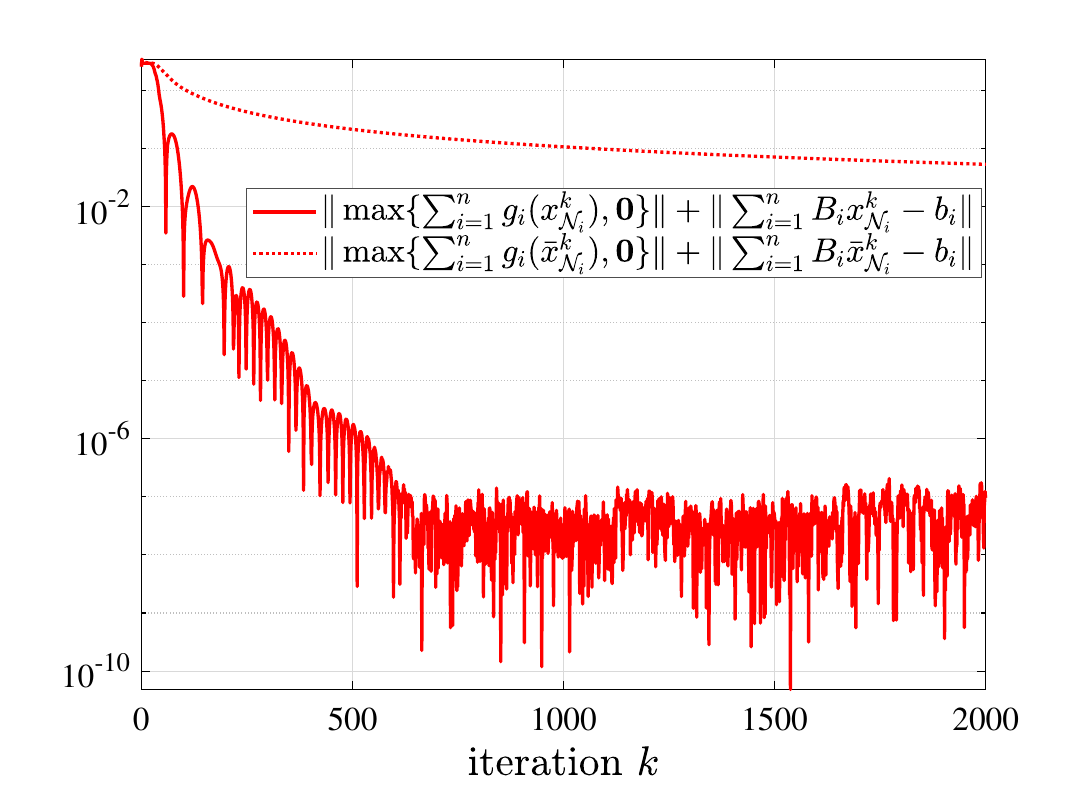}}
	\caption{Convergence performance of Algorithm~\ref{alg:algorithm} for solving problem~\eqref{eq:numercialexamp2}.}  
	\label{ex2:optimal_error}
\end{figure}

Problem \eqref{eq:numercialexamp2} has a complicated coupling structure in the form of \eqref{eq:primalprob} and no existing decentralized optimization algorithms can solve \eqref{eq:numercialexamp2}. Hence, we only execute Algorithm~\ref{alg:algorithm} to solve it. Fig.~\ref{ex2:optimal_error} depicts the objective error and the sum of inequality constraint violation and equality constraint violation at $\mathbf{x}^k$ and $\bar{\mathbf{x}}^k$, respectively. The results in Fig.~\ref{ex2:optimal_error} verify the effectiveness of Algorithm~\ref{alg:algorithm} on solving problems with both variable coupling and globally-coupled constraints. 

\section{Conclusion}\label{sec:conclusion}
In this paper, we have developed a decentralized projected primal-dual algorithm for addressing a generally-coupled constrained optimization problem with a variable coupling structure, where the local objective and constraint functions of each node are determined by local variables as well as the variables of its neighbors. The proposed algorithm is developed based on the gradient projection method, a virtual queue technique, and a primal-dual-primal method, and is shown to achieve convergence at $O(1/k)$ rates in terms of both optimality and infeasibility. The effectiveness and efficiency of the proposed algorithm have been demonstrated via two sets of simulations.

\appendix

\subsection{Proof of Proposition~\ref{pro:propertyofproby}}\label{ssec:proofofproLipschitz}

First, the convexity of $\mathbf{f}$ and $\mathbf{G}$ in (\ref{pro:fygyconvex}) is a direct consequence of the convexity of $\sum_{i \in \mathcal{V}} f_{i}\left(x_{\mathcal{N}_{i}}\right)$ and $g_i$ in Assumption~\ref{asm:primalprob}(\ref{slater}). 

To prove the Lipschitz continuity in (\ref{pro:gardfsmooth}) and (\ref{pro:Gycontinuous}), let $\mathbf{y}=[x_1^T,t_1^T,\ldots,x_n^T,t_n^T]^T\in Y$ and $\mathbf{y}'=[(x_1')^T,(t_1')^T,\ldots,(x_n')^T,(t_n')^T]^T\in Y$. For simplicity, for any $j\in\mathcal{V}$ and $i\in\mathcal{N}_j$, let $\nabla_i f_j(x_{\mathcal{N}_{j}})$ be the partial derivative of $f_j(x_{\mathcal{N}_{j}})$ with respect to $x_i$. Then, due to Assumption~\ref{asm:primalprob}(\ref{fismooth}), $\|\nabla\mathbf{f}(\mathbf{y})-\nabla\mathbf{f}(\mathbf{y'})\|^2=\sum_{i\in \mathcal{V}}\|\sum_{j \in  \mathcal{N}_{i}}(\nabla_i f_j(x_{\mathcal{N}_{j}})-\nabla_i f_j(x'_{\mathcal{N}_{j}}))\|^2\le\max_{i\in \mathcal{V}}(|\mathcal{N}_{i}|)\sum_{i\in \mathcal{V}}\sum_{j\in\mathcal{N}_{i}}\|\nabla_i f_j(x_{\mathcal{N}_{j}})-\nabla_i f_j(x'_{\mathcal{N}_{j}})\|^2=\max_{i\in \mathcal{V}}(|\mathcal{N}_{i}|)\sum_{i\in \mathcal{V}} \|\nabla f_i(x_{\mathcal{N}_{i} })
\linebreak[4]-\nabla f_i(x'_{\mathcal{N}_{i}})\|^2\le \max_{i\in \mathcal{V}}(|\mathcal{N}_{i}|)L_f^2\sum_{i\in \mathcal{V}}\|x_{\mathcal{N}_{i} }-x_{\mathcal{N}_{i} }'\|^2$. In addition, we have
\begin{align}
&\sum_{i\in \mathcal{V}}\|x_{\mathcal{N}_{i} }-x_{\mathcal{N}_i}'\|^2 = \sum_{i\in \mathcal{V}}\sum\limits_{j \in \mathcal{N}_i }\|x_j-x_j'\|^2\nonumber\displaybreak[0]\\
&\le\max_{i\in \mathcal{V}}(|\mathcal{N}_i|)\sum_{i\in \mathcal{V}}\|x_i-x_i'\|^2.\label{eq:relationshipxNixi}
\end{align}
It follows that $\|\nabla\mathbf{f}(\mathbf{y})-\nabla\mathbf{f}(\mathbf{y'})\|^2 \le L_f^2 (\max_{i\in \mathcal{V}}(|\mathcal{N}_{i}|))^2\sum_{i\in \mathcal{V}}\|x_i-x_i'\|^2\le L_f^2(\max_{i\in \mathcal{V}}(|\mathcal{N}_{i}|))^2\|\mathbf{y}-\mathbf{y'}\|^2$, i.e., property~(\ref{pro:gardfsmooth}) holds.

From Assumption~\ref{asm:primalprob}(\ref{giLipschitz}) and \eqref{eq:relationshipxNixi},
\begin{align*}
&\|\mathbf{G}(\mathbf{y})-\mathbf{G}(\mathbf{y}')\|^2\displaybreak[0]\\
=&\sum_{i\in \mathcal{V}}\|(g_i(x_{\mathcal{N}_{i}})-t_i)-(g_i(x_{\mathcal{N}_{i}}')-t_i')\|^2\displaybreak[0]\\
\le&\sum_{i\in \mathcal{V}}(1+\frac{1}{\beta^2})\|g_i(x_{\mathcal{N}_{i}})-g_i(x_{\mathcal{N}_{i}}')\|^2\displaybreak[0]\\
&+\sum_{i\in \mathcal{V}}(1+\beta^2)\|t_i-t_i'\|^2\displaybreak[0]\\
\le&\sum_{i\in \mathcal{V}}(1+\beta^2)(\|x_{\mathcal{N}_{i}}-x_{\mathcal{N}_{i}}'\|^2+\|t_i-t_i'\|^2)\displaybreak[0]\\
\le&(1+\beta^2)\max_{i\in \mathcal{V}}(|\mathcal{N}_{i}|)\|\mathbf{y}-\mathbf{y'}\|^2,
\end{align*}
where the second step comes from $-2\langle g_i(x_{\mathcal{N}_{i}})-g_i(x_{\mathcal{N}_{i}}'), t_i-t_i'\rangle\le \frac{1}{\beta^2}\|g_i(x_{\mathcal{N}_{i}})-g_i(x_{\mathcal{N}_{i}}')\|^2+\beta^2\|t_i-t_i'\|^2$ $\forall i\in\mathcal{V}$. Therefore, property~(\ref{pro:Gycontinuous}) holds.

To prove property~(\ref{pro: equivsolguarteen}), let $\mathbf{x}^\star=[(x_1^\star)^T, \ldots, (x_n^\star)^T]^T$ be an optimum of problem~\eqref{eq:primalprob}, which is guaranteed to exist due to Assumption~\ref{asm:primalprob}(\ref{least1soluprimal}). Then, we let $t_i^\star=g_i(x_{\mathcal{N}_{i}}^\star)-\frac{1}{n}\sum_{j\in \mathcal{V}} g_j(x_{\mathcal{N}_{j}}^\star)$ $\forall i\in \mathcal{V}$, and let $\mathbf{y}^\star=[(x_1^\star)^T, (t_1^\star)^T,\ldots,(x_n^\star)^T, (t_n^\star)^T]^T$. It can be shown that $\mathbf{y}^\star$ is feasible to problem~\eqref{eq:probleminy}. Moreover, because $\mathbf{f}(\mathbf{y}^\star)=\sum_{i\in \mathcal{V}} f_i(x_{\mathcal{N}_{i}}^\star)$ and because \eqref{eq:primalprob} and \eqref{eq:probleminy} share the same optimal value, $\mathbf{y}^\star$ is an optimum of problem~\eqref{eq:probleminy}. On the other hand, let  $\mathbf{y}^\star=[(x_1^\star)^T, (t_1^\star)^T,\ldots,(x_n^\star)^T, (t_n^\star)^T]^T$ be an optimum of problem~\eqref{eq:probleminy}. We can verify that $[(x_1^\star)^T, \ldots, (x_n^\star)^T]^T$ is feasible to problem~\eqref{eq:primalprob}, where the coupled equality constraint in \eqref{eq:primalprob} is directly satisfied and the coupled inequality constraint $\sum_{i=1}^n g_i(x_{\mathcal{N}_{i}}^\star) \le \mathbf{0}_p$ is obtained by summing $\sum_{i=1}^n t_i^\star=\mathbf{0}_p $ and each $ g_i(x_{\mathcal{N}_{i}}^\star)-t_i^\star \le \mathbf{0}_p$ in \eqref{eq:probleminy}. Moreover, $\sum_{i\in \mathcal{V}} f_i(x_{\mathcal{N}_{i}}^\star)=\mathbf{f}(\mathbf{y}^\star)$. Therefore, $[(x_1^\star)^T, \ldots, (x_n^\star)^T]^T$ is an optimum of problem~\eqref{eq:primalprob} and property~\eqref{pro: equivsolguarteen} holds.

Finally, let $\tilde{x}_i$ $\forall i\in\mathcal{V}$ be given by Assumption~\ref{asm:primalprob}(\ref{slater}), and let $\tilde{t}_i=g_i(\tilde{x}_{\mathcal{N}_{i}})-\frac{1}{n}\sum_{j\in \mathcal{V}} g_j(\tilde{x}_{\mathcal{N}_{j}})$ $\forall i\in \mathcal{V}$. It can thus be verified that $[(\tilde{x}_1)^T, (\tilde{t}_1)^T,\ldots,(\tilde{x}_n)^T, (\tilde{t}_n)^T]^T$ is a Slater's point for problem~\eqref{eq:probleminy}, so that property~(\ref{pro:equivstrongdual}) holds.

\subsection{Proof of Theorem~\ref{theo:feasibility}} \label{ssec:proofoftheorem1}

We first bound the constraint violations with respect to problem~\eqref{eq:probleminy} at $\bar{\mathbf{y}}^k$ in the lemma below.

\begin{lemma}\label{lemma:GybarBybarboundk}
Suppose Assumptions~\ref{asm:primalprob} and~\ref{asm:matricesphpw} hold. Then, for each $k\ge 1$,
\begin{align}
	\mathbf{G}(\bar{\mathbf{y}}^k)&\le\frac{\mathbf{q}^{k}}{k}, \label{eq:Gybarbound}\displaybreak[0]\\
	\!\!\!(\mathbf{1}_n\otimes I_{m+p})^T(\mathbf{B}\bar{\mathbf{y}}^k-\mathbf{c})&=\frac{\rho}{k}(\mathbf{1}_n\otimes I_{m+p})^T(\mathbf{u}^{k}-\mathbf{u}^{0}).\label{eq:Bybarbound}
\end{align}
\end{lemma}

\begin{proof}
Refer to the proof of \cite[eq.(37), eq.(38)] {wu2022distributed}.	
\end{proof}		

From Lemma~\ref{lemma:GybarBybarboundk}, Theorem~\ref{theo:feasibility} can be proved as long as we can further bound $\mathbf{q}^{k}$ and $(\mathbf{1}_n\otimes I_{m+p})^T(\mathbf{u}^{k}-\mathbf{u}^{0})$. To this end, we introduce the following lemmas.

\begin{lemma} \label{lemma:lipschitzofR}
Suppose Assumptions~\ref{asm:primalprob} and~\ref{asm:matricesphpw} hold. For each $k\ge0$, $\nabla\mathcal{R}^k(\mathbf{y})$ is Lipschitz continuous on the set $Y$ with Lipschitz constant $L_{\mathcal{R}^k}:=L_F+\langle\mathbf{q}^{k}+\mathbf{G}(\mathbf{y}^{k}), L_g\mathbf{1}_{np}\rangle+\frac{1}{\rho}\|\mathbf{B}^{T}\mathbf{B}\|$.
\end{lemma}

\begin{proof}
For any $\mathbf{y},\mathbf{y'}\in Y$,
\begin{align}
	&\|\nabla \mathcal{R}^k(\mathbf{y})-\nabla \mathcal{R}^k(\mathbf{y'})\|\nonumber\displaybreak[0]\\
	&\le\|\nabla\mathbf{f}(\mathbf{y})-\nabla\mathbf{f}(\mathbf{y'})\|+\|\frac{1}{\rho}\mathbf{B}^{T}\mathbf{B}(\mathbf{y}-\mathbf{y'})\|\!\nonumber\displaybreak[0]\\
	&+\sum_{i=1}^{n}\sum\limits_{\ell=1}^{p}\|(q_{i\ell}^k+G_{i\ell}(\mathbf{y}^{k}))(\nabla G_{i\ell}(\mathbf{y})-\nabla G_{i\ell}(\mathbf{y'}))\|,\label{eq:smoothofR}
\end{align}
where $q_{i\ell}^k$ and $G_{i\ell}$, $\ell=1,\ldots,p$ represent the $\ell$th coordinates of $\mathbf{q}^k$ and $G_i$, respectively. Note from $\mathbf{q}^{0} = \max\{-\mathbf{G}(\mathbf{y}^{0}),\mathbf{0}\}$ and \eqref{eq:updateofqk} that $\mathbf{q}^{k}+\mathbf{G}(\mathbf{y}^{k})\ge\mathbf{0}$. It follows from Assumption~\ref{asm:primalprob}(\ref{gijsmooth})	that $\|(q_{i\ell}^k+G_{i\ell}(\mathbf{y}^{k}))(\nabla G_{i\ell}(\mathbf{y})-\nabla G_{i\ell}(\mathbf{y'}))\|=(q_{i\ell}^k+G_{i\ell}(\mathbf{y}^{k}))\|\nabla G_{i\ell}(\mathbf{y})-\nabla G_{i\ell}(\mathbf{y'})\|\le(q_{i\ell}^k+G_{i\ell}(\mathbf{y}^{k}))L_g\|\mathbf{y}-\mathbf{y'}\|$. Incorporating this and Proposition~\ref{pro:propertyofproby}(\ref{pro:gardfsmooth}) into \eqref{eq:smoothofR} leads to $\|\nabla \mathcal{R}^k(\mathbf{y})-\nabla \mathcal{R}^k(\mathbf{y'})\|\le L_{\mathcal{R}^k}\|\mathbf{y}-\mathbf{y'}\|$.
\end{proof}

\begin{lemma}\label{lemma:qkproperties} 
For each $k\ge 0$,
\begin{align}
	&\mathbf{q}^{k}\ge\mathbf{0},\label{eq:qknonnegative}\displaybreak[0]\\
	&\|\mathbf{q}^{0}\| \le \|\mathbf{G}(\mathbf{y}^{0})\|,\quad\|\mathbf{q}^{k+1}\| \ge \|\mathbf{G}(\mathbf{y}^{k+1})\|.\label{eq:qkplus1normlargethang}
\end{align} 
\end{lemma}

\begin{proof}
It follows from \cite[Lemma~3 and Lemma~4] {YuHao2017}.
\end{proof}

\begin{lemma}\label{lemma:boundedsumF}
Suppose all the conditions in Theorem~\ref{theo:feasibility} hold. If there is some $K\ge1$ such that $\|\mathbf{q}^{k}\|\le 2(\|\bm{\lambda}^\star\|+\sqrt{C})$ $\forall k=0,1, 2,\ldots, K-1$, then,
\begin{equation}\label{eq:sumPhiykystarupperbound}
	\sum_{k=1}^{K} (\mathbf{f}(\mathbf{y}^{k})-\mathbf{f}(\mathbf{y}^\star)) \le S^{0}-S^{K},
\end{equation}
where $S^k=\frac{1}{2\rho}\|\mathbf{z}^{k}-\mathbf{z}^\star\|_{H^\dag}^2+\frac{\rho}{2}\|\mathbf{u}^{k}\|_W^2+\frac{1}{2}\|\mathbf{y}^{k}-\mathbf{y}^\star\|^2_{\frac{I}{\gamma}-\frac{\mathbf{B^T}\mathbf{B}}{\rho}}+\frac{1}{2}\|\mathbf{q}^{k}\|^2-\frac{1}{2}\|\mathbf{G}(\mathbf{y}^{k})\|^2$ $\forall k\ge 0$. 
\end{lemma}

\begin{proof}
To prove \eqref{eq:sumPhiykystarupperbound}, we first bound $\mathbf{f}(\mathbf{y}^{k+1})-\mathbf{f}(\mathbf{y}^\star)$ $\forall k\ge0$. Recall from Section~\ref{sec:algdevelop} that $\mathbf{d}^{k}$ is the gradient of $\mathcal{R}^k(\mathbf{y})$ at $\mathbf{y}^k$. Thus, \eqref{eq:projofyk} can be rewritten as $\mathbf{y}^{k+1}=\operatorname{arg\;min}_{\mathbf{y} \in  Y}M^k(\mathbf{y})$, where $M^k(\mathbf{y})=\mathcal{R}^k(\mathbf{y}^{k})+\nabla \mathcal{R}^k(\mathbf{y}^{k})^{T}(\mathbf{y}-\mathbf{y}^{k})+\frac{1}{2 \gamma}\|\mathbf{y}-\mathbf{y}^{k}\|^{2}$. Since $M^k(\mathbf{y})$ is $\frac{1}{\gamma}$-strongly convex and $\mathbf{y}^{k+1}$ minimizes $M^k(\mathbf{y})$, we have $M^k(\mathbf{y}^\star)\ge M^k(\mathbf{y}^{k+1}) +\frac{1}{2\gamma}\|\mathbf{y}^{k+1}-\mathbf{y}^{\star}\|^2$, which is equivalent to
\begin{align}
	\langle \nabla\mathcal{R}^k(\mathbf{y}^{k}), \mathbf{y}^{k+1}-\mathbf{y}^\star\rangle&\le \frac{1}{2\gamma}(\|\mathbf{y}^k-\mathbf{y}^{\star}\|^2-\|\mathbf{y}^{k+1}-\mathbf{y}^\star\|^2)\nonumber\displaybreak[0]\\
	&-\frac{1}{2\gamma}\|\mathbf{y}^{k+1}-\mathbf{y}^{k}\|^2.\label{eq:substiexpreoftildeMk}
\end{align} 
On the other hand, from Lemma~\ref{lemma:lipschitzofR},
\begin{align}
	\mathcal{R}^k(\mathbf{y}^{k+1})\le &\mathcal{R}^k(\mathbf{y}^{k})+\langle \nabla \mathcal{R}^k (\mathbf{y}^{k}), \mathbf{y}^{k+1}-\mathbf{y}^{k}\rangle\nonumber\displaybreak[0]\\
	&+\frac{L_{\mathcal{R}^k}}{2}\|\mathbf{y}^{k+1}-\mathbf{y}^{k}\|^2.\label{eq:usesmoothofR}
\end{align}
In addition, due to the convexity of $\mathcal{R}^k(\mathbf{y})-\frac{1}{2\rho}\|\mathbf{B}\mathbf{y}-\mathbf{c}\|^2$, we obtain
\begin{align*}
	&-\mathcal{R}^k(\mathbf{y}^\star)+\frac{1}{2\rho}\|\mathbf{B}\mathbf{y}^\star-\mathbf{c}\|^2\le -\mathcal{R}^k(\mathbf{y}^{k})+\frac{1}{2\rho}\|\mathbf{B}\mathbf{y}^{k}-\mathbf{c}\|^2\displaybreak[0]\\
	&+\langle \nabla\mathcal{R}^k (\mathbf{y}^{k})-\frac{1}{\rho}\mathbf{B}^{T}(\mathbf{B}\mathbf{y}^{k}-\mathbf{c}), \mathbf{y}^{k}-\mathbf{y}^\star\rangle.
\end{align*}
By adding the above inequality and \eqref{eq:usesmoothofR}, we have 
\begin{align*}
	&\mathcal{R}^k(\mathbf{y}^{k+1}) -\mathcal{R}^k(\mathbf{y}^\star)
	\le\langle \nabla\mathcal{R}^k(\mathbf{y}^{k}), \mathbf{y}^{k+1}\!-\!\mathbf{y}^\star\rangle\nonumber\displaybreak[0]\\
	&+\frac{L_{\mathcal{R}^k}}{2}\|\mathbf{y}^{k+1}-\mathbf{y}^{k}\|^2-\frac{1}{2\rho}\|\mathbf{B}\mathbf{y}^k-\mathbf{B}\mathbf{y}^\star\|^2.
\end{align*}
Let $M_1^k(\mathbf{y})=\langle W\mathbf{u}^{k}\!-\frac{1}{\rho}\mathbf{z}^{k}, \mathbf{B}\mathbf{y}\!-\!\mathbf{c}\rangle+\frac{1}{2\rho}\|\mathbf{B}\mathbf{y}-\mathbf{c}\|^2\!$, $
M_2^k(\mathbf{y})=\langle \mathbf{q}^{k}+\mathbf{G}(\mathbf{y}^{k}), \mathbf{G}(\mathbf{y})\rangle$ and we have $\mathcal{R}^k(\mathbf{y})=\mathbf{f}(\mathbf{y})+M_1^k(\mathbf{y})+M_2^k(\mathbf{y})$. Substituting this definition into the combination of the above inequality and \eqref{eq:substiexpreoftildeMk} yields
\begin{align}
	&\mathbf{f}(\mathbf{y}^{k+1})-\!\mathbf{f}(\mathbf{y}^\star)\nonumber\displaybreak[0]\\
	&\le M_1^k(\mathbf{y}^\star)-M_1^k(\mathbf{y}^{k+1})+M_2^k(\mathbf{y}^\star)-M_2^k(\mathbf{y}^{k+1})\nonumber\displaybreak[0]\\
	&+\!(\frac{L_{\mathcal{R}^k}}{2}\!-\frac{1}{2\gamma}\!)\|\mathbf{y}^{k+1}\!\!-\!\mathbf{y}^{k}\|^2\!-\!\frac{1}{2\rho}\|\mathbf{B}\mathbf{y}^k-\mathbf{B}\mathbf{y}^\star\|^2\nonumber\displaybreak[0]\\
	&+\frac{1}{2\gamma}(\|\mathbf{y}^k\!-\!\mathbf{y}^{\star}\|^2-\|\mathbf{y}^{k+1}\!\!-\mathbf{y}^\star\|^2).\label{eq:fyfstarbound}
\end{align}
By referring to the proof of \cite[Lemma 2] {wu2022distributed}, we are able to derive the bounds on $M_1^k(\mathbf{y}^\star)-M_1^k(\mathbf{y}^{k+1})$ and $M_2^k(\mathbf{y}^\star)-M_2^k(\mathbf{y}^{k+1})$ in \eqref{eq:fyfstarbound} as follows:
\begin{align*}
	&M_1^k(\mathbf{y}^\star)-M_1^k(\mathbf{y}^{k+1})\displaybreak[0]\\
	&\le\frac{1}{2\rho}(\|\mathbf{z}^\star-\mathbf{z}^{k}\|_{H^\dag}^2-\|\mathbf{z}^\star-\mathbf{z}^{k+1}\|_{H^\dag}^2)\displaybreak[0]\\
	&+\frac{\rho}{2}(\|\mathbf{u}^{k}\|_W^2-\|\mathbf{u}^{k+1}\|_W^2)+\frac{1}{2\rho}\|\mathbf{B}\mathbf{y}^\star-\mathbf{B}\mathbf{y}^{k+1}\|^2,\displaybreak[0]\\
	&M_2^k(\mathbf{y}^\star)\!-\!M_2^k(\mathbf{y}^{k+1})\displaybreak[0]\\
	&\le\frac{1}{2}(\|\mathbf{q}^{k}\|^2\!-\!\|\mathbf{q}^{k+1}\|^2)+\!\frac{1}{2}(\|\mathbf{G}(\mathbf{y}^{k+1})\|^2-\|\mathbf{G}(\mathbf{y}^{k})\|^2)\displaybreak[0]\\
	&+\frac{\tilde{\beta}^2}{2}\|\mathbf{y}^{k+1}-\mathbf{y}^{k}\|^2.
\end{align*}
Substituting the above two inequalities into \eqref{eq:fyfstarbound} gives
\begin{align}
	&\mathbf{f}(\mathbf{y}^{k+1})-\mathbf{f}(\mathbf{y}^\star) \le S^k-S^{k+1}\nonumber\displaybreak[0]\\
	&+\frac{1}{2}(\tilde{\beta}^2+L_{\mathcal{R}^k}-\frac{1}{\gamma})\|\mathbf{y}^{k+1}-\mathbf{y}^{k}\|^2.\label{eq: fyk1fystarbound}
\end{align}
Next, note from Lemma~\ref{lemma:lipschitzofR} that
\begin{align*}
	&L_{\mathcal{R}^k}\le L_F+\frac{1}{\rho}\|\mathbf{B}^{T}\mathbf{B}\|+\sqrt{np}(\|\mathbf{q}^{k}\|+\|\mathbf{G}(\mathbf{y}^{k})\|)L_g\displaybreak[0]\\
	&\le L_F+\frac{1}{\rho}\|\mathbf{B}^{T}\mathbf{B}\|+2\sqrt{np}\max\{\|\mathbf{q}^{k}\|,\|\mathbf{G}(\mathbf{y}^{k})\|\}L_g.
\end{align*}
Due to \eqref{eq:qkplus1normlargethang}, $\|\mathbf{G}(\mathbf{y}^{0})\|\le 2(\|\bm{\lambda}^\star\|+\sqrt{C})$, and the hypothesis of Lemma~\ref{lemma:boundedsumF}, we have 
\begin{align*}
	L_{\mathcal{R}^k}\le L_F+\frac{1}{\rho}\|\mathbf{B}^{T}\mathbf{B}\|+ 4\sqrt{np}(\|\bm{\lambda}^\star\|+\sqrt{C})L_g,
\end{align*}
for each $k=0,1,\ldots,K-1$. This, along with Lemma~\ref{lemma:lipschitzofR} and the parameter condition in Theorem~\ref{theo:feasibility}, gives $\frac{1}{\gamma}\ge\tilde{\beta}^2+L_{\mathcal{R}^k}$ $\forall k=0,\ldots,K-1$. It then follows from \eqref{eq: fyk1fystarbound} that $\mathbf{f}(\mathbf{y}^{k+1})-\mathbf{f}(\mathbf{y}^\star)\le S^k-S^{k+1}$ $\forall k=0,1,\ldots, K-1$, leading to \eqref{eq:sumPhiykystarupperbound}.	  
\end{proof}

Based on the above lemmas, we provide the bounds on $\|\mathbf{q}^{k}\|$ and $\|\mathbf{u}^{k}-\mathbf{u}^{0}\|_W$ using mathematical induction. 

\begin{lemma}\label{lemma:qkuku0wbound}
Suppose all conditions in Theorem~\ref{theo:feasibility} hold. For each $k\ge 0$,
\begin{align}
	&\|\mathbf{q}^{k}\|\le 2(\|\bm{\lambda}^\star\|+\sqrt{C}),\label{eq:qbound}\\
	&\|\mathbf{u}^{k}-\mathbf{u}^{0}\|_W\le\|\mathbf{u}^{0}\|_W+\|\mathbf{u}^\star\|_W+\sqrt{\frac{2C}{\rho}}.\label{eq:ubound}
\end{align}
\end{lemma}

\begin{proof}
We first prove \eqref{eq:qbound} for all $k\ge0$ by induction. Due to \eqref{eq:qkplus1normlargethang}, $\|\mathbf{q}^{0}\| \le \|\mathbf{G}(\mathbf{y}^{0})\| \le 2\sqrt{C}$, so that \eqref{eq:qbound} is satisfied when $k=0$. Next, given $K\ge1$, suppose \eqref{eq:qbound} holds for all $k=0,1,\ldots,K-1$, and below we show \eqref{eq:qbound} holds for $k=K$.

To do so, note from Proposition~\ref{pro:propertyofproby}(\ref{pro:equivstrongdual}) that $\mathbf{f}(\mathbf{y}^\star)= \operatorname{\min}_{ \mathbf{y} \in  Y}~\mathbf{f}(\mathbf{y})+\langle \bm{\lambda}^\star,\mathbf{G}(\mathbf{y})\rangle+\langle \mathbf{u}^\star, \mathbf{B}\mathbf{y}-\mathbf{c}\rangle$, implying that $\mathbf{f}(\mathbf{y}^\star)\le\mathbf{f}(\mathbf{y}^{k})+\langle \bm{\lambda}^\star,\mathbf{G}(\mathbf{y}^{k})\rangle+\langle \mathbf{u}^\star, \mathbf{B}\mathbf{y}^{k}-\mathbf{c}\rangle$ $\forall k\ge 1$. Adding this inequality over $k=1,\ldots,K$ yields
\begin{align}
	&\sum_{k=1}^{K}(\mathbf{f}(\mathbf{y}^\star)-\mathbf{f}(\mathbf{y}^{k}))\nonumber\displaybreak[0]\\
	&\le\langle \bm{\lambda}^\star,\sum_{k=1}^{K}\mathbf{G}(\mathbf{y}^{k}\rangle+\langle \mathbf{u}^\star, \sum_{k=1}^{K}(\mathbf{B}\mathbf{y}^{k}-\mathbf{c})\rangle.\label{eq:sumfstarfykupperbound}
\end{align}
By referring to the proof of \cite[eq.(40), eq.(41)] {wu2022distributed}, we can show that $\sum_{k=1}^{K} \mathbf{G}(\mathbf{y}^{k})\le \mathbf{q}^{K}$ and $\sum_{k=1}^{K}(\mathbf{1}_n\otimes I_{m+p})^T(\mathbf{B}\mathbf{y}^{k}-\mathbf{c})=\rho(\mathbf{1}_n\otimes I_{m+p})^T(\mathbf{u}^{K}-\mathbf{u}^{0})$. Also note that $\bm{\lambda}^\star\ge\mathbf{0}$ and $\mathbf{u}^\star\in \operatorname{Range}(\mathbf{1}_n\otimes I_{m+p})$. It follows from \eqref{eq:sumfstarfykupperbound} that
\begin{equation*}
	\sum_{k=1}^{K}(\mathbf{f}(\mathbf{y}^\star)-\mathbf{f}(\mathbf{y}^{k}))\le \langle \bm{\lambda}^\star,\mathbf{q}^{K}\rangle+\rho \langle \mathbf{u}^\star, \mathbf{u}^{K}-\mathbf{u}^{0}\rangle.
\end{equation*}
By combining the above inequality with \eqref{eq:sumPhiykystarupperbound} and by using $W\mathbf{u}^\star = \mathbf{u}^\star$ and $\langle W\mathbf{u}^\star, \mathbf{u}^{K}-\mathbf{u}^{0}\rangle =\langle W^{\frac{1}{2}}\mathbf{u}^\star, W^{\frac{1}{2}}(\mathbf{u}^{K}-\mathbf{u}^{0})\rangle$,
\begin{align}
	&S^{K}\le S^{0}\!+\!\langle \bm{\lambda}^\star,\mathbf{q}^{K}\rangle+\rho\langle W\mathbf{u}^\star, \mathbf{u}^{K}-\mathbf{u}^{0}\rangle\nonumber\\
	&\le S^{0}+\|\bm{\lambda}^\star\|\cdot\|\mathbf{q}^{K}\|+\rho\|\mathbf{u}^\star\|_W(\|\mathbf{u}^{K}\|_W+\|\mathbf{u}^{0}\|_W).\label{eq:Rkupperboundcrossterms}
\end{align}
In addition, following the proof of \cite[Theorem 1] {wu2022distributed}, we can show via Proposition~\ref{pro:propertyofproby}(\ref{pro:Gycontinuous}) that $\|\mathbf{G}(\mathbf{y}^{K})\|^2\le \frac{1}{2}\|\mathbf{q}^{K}\|^2+\tilde{\beta}^2\|\mathbf{y}^{K}-\mathbf{y}^\star\|^2+\|\mathbf{G}(\mathbf{y}^\star)\|^2$. By incorporating this into the expression of $S^K$ given in Lemma~\ref{lemma:boundedsumF} and because $\frac{1}{\gamma}\ge \tilde{\beta}^2+\frac{1}{\rho}\|\mathbf{B}^{T}\mathbf{B}\|$, 
\begin{equation}\label{eq:Rklowerbound}
	S^{K}\ge\frac{\rho}{2}\|\mathbf{u}^{K}\|_W^2+\frac{1}{4}\|\mathbf{q}^{K}\|^2-\frac{1}{2}\|\mathbf{G}(\mathbf{y}^\star)\|^2.
\end{equation}
Thus, \eqref{eq:Rkupperboundcrossterms} and \eqref{eq:Rklowerbound} together result in 
\begin{equation} \label{eq:constraintboundforK}
	\frac{\rho}{2}(\|\mathbf{u}^{K}\|_W-\|\mathbf{u}^\star\|_W)^2+\frac{1}{4}(\|\mathbf{q}^{K}\|-2\|\bm{\lambda}^\star\|)^2\le C.
\end{equation}
It can be seen that $\|\mathbf{q}^{K}\| \le  2(\|\bm{\lambda}^\star\|+\sqrt{C})$. This completes the proof of \eqref{eq:qbound} for each $k\ge0$.

Finally, note from \eqref{eq:constraintboundforK} that $\|\mathbf{u}^{K}-\mathbf{u}^{0}\|_W\le\|\mathbf{u}^{0}\|_W+\|\mathbf{u}^\star\|_W+(\|\mathbf{u}^{K}\|_W-\|\mathbf{u}^\star\|_W)\le\|\mathbf{u}^{0}\|_W+\|\mathbf{u}^\star\|_W+\sqrt{\frac{2C}{\rho}}$. Since $K$ can be any positive integer and \eqref{eq:ubound} is satisfied when $k=0$, we conclude that \eqref{eq:ubound} holds for all $k\ge0$.
\end{proof}

From \eqref{eq:qbound}, $\mathbf{q}^{k} \le 2(\|\bm{\lambda}^\star\|+\sqrt{C})\mathbf{1}_{np}$ $\forall k \ge 0$, which, together with \eqref{eq:Gybarbound}, leads to \eqref{eq:Gybarkupperbound}. In addition, since $P^W=(P^W)^T \succeq \mathbf{O}$ and $P^W\mathbf{1}_n=\mathbf{1}_n$, we have $P^W-\frac{\mathbf{1}_n\mathbf{1}_n^T}{n} \succeq \mathbf{O}$. Hence, $\|(\mathbf{1}_n\otimes I_{m+p})^T(\mathbf{u}^{k}-\mathbf{u}^{0})\|^2\le n(\mathbf{u}^{k}-\mathbf{u}^{0})^T (P^W \otimes I_{m+p}) (\mathbf{u}^{k}-\mathbf{u}^{0})=n\|\mathbf{u}^{k}-\mathbf{u}^{0}\|^2_W$. It follows from \eqref{eq:Bybarbound} and \eqref{eq:ubound} that \eqref{eq:Bybarcupperbound} holds.

\subsection{Proof of Theorem~\ref{theo:funcval}}\label{proofoftheorem2}
Let $k\ge1$. Because $\frac{I}{\gamma}-\frac{\mathbf{B^T}\mathbf{B}}{\rho} \succ \mathbf{O}$ and because of \eqref{eq:qkplus1normlargethang}, we have $S^k\ge 0$, where $S^k$ is defined in Lemma~\ref{lemma:boundedsumF}. This, along with \eqref{eq:sumPhiykystarupperbound}, results in
$\sum_{t=1}^k (\mathbf{f}(\mathbf{y}^{t})-\mathbf{f}(\mathbf{y}^\star)) \le S^{0}$. Moreover, due to the convexity of $\mathbf{f}$, $\mathbf{f}(\bar{\mathbf{y}}^k)-\mathbf{f}(\mathbf{y}^\star)\le\frac{1}{k}\sum_{t=1}^k(\mathbf{f}(\mathbf{y}^{t})-\mathbf{f}(\mathbf{y}^\star))$. Therefore, $\mathbf{f}(\bar{\mathbf{y}}^k)-\mathbf{f}(\mathbf{y}^\star)\le\frac{S^{0}}{k}$.

On the other hand, since $\mathbf{f}(\mathbf{y}^\star) =\min_{\mathbf{y}\in  Y}~\mathbf{f}(\mathbf{y})+\langle \bm{\lambda}^\star,\mathbf{G}(\mathbf{y})\rangle+\langle u^\star, (\mathbf{1}_n\otimes I_{m+p})^T(\mathbf{B}\mathbf{y}-\mathbf{c})\rangle$,
\begin{align}
\mathbf{f}(\bar{\mathbf{y}}^k)-\mathbf{f}(\mathbf{y}^\star)\ge&-\langle \bm{\lambda}^\star,\mathbf{G}(\bar{\mathbf{y}}^k)\rangle
\nonumber\displaybreak[0]\\
&-\langle u^\star, (\mathbf{1}_n\otimes I_{m+p})^T(\mathbf{B}\bar{\mathbf{y}}^k-\mathbf{c})\rangle.\label{eq:PhibarykminusPhiystar}
\end{align}
Because $\bm{\lambda}^\star\ge \mathbf{0}$ and because of \eqref{eq:Gybarkupperbound},
\begin{align}
\langle\bm{\lambda}^\star,\mathbf{G}(\bar{\mathbf{y}}^k)\rangle\le \frac{2(\|\bm{\lambda}^\star\|+\sqrt{C})\mathbf{1}_{np}^T\bm{\lambda}^\star}{k}.\label{eq:QstarGbaryk}
\end{align}
Additionally, since $\mathbf{u}^\star=\mathbf{1}_n\otimes u^\star$ and due to \eqref{eq:Bybarcupperbound},
\begin{align}
&\langle u^\star, (\mathbf{1}_n\otimes I_{m+p})^T(\mathbf{B}\bar{\mathbf{y}}^k-\mathbf{c})\rangle\nonumber\displaybreak[0]\\
\le&\|u^\star\|\cdot\|(\mathbf{1}_n\otimes I_{m+p})^T(\mathbf{B}\bar{\mathbf{y}}^k-\mathbf{c})\|\nonumber\displaybreak[0]\\
\le&\frac{\rho\|\mathbf{u}^\star\|(\|\mathbf{u}^{0}\|_W+\|\mathbf{u}^\star\|_W+\sqrt{2C/\rho})}{k}\label{eq:ustarBbarykc}.
\end{align}
Incorporating \eqref{eq:QstarGbaryk} and \eqref{eq:ustarBbarykc} into \eqref{eq:PhibarykminusPhiystar} yields $\mathbf{f}(\bar{\mathbf{y}}^k)-\mathbf{f}(\mathbf{y}^\star)\ge-\frac{C^0}{k}$. This completes the proof.

\subsection{Proof of Corollary~\ref{theo:originalprobconv}}\label{proofoftheorem3}
Let $k\ge 1$. From the definitions of $\mathbf{B}$ and $\mathbf{c}$ in Section~\ref{sec:algdevelop}, 
\begin{align}
&\|(\mathbf{1}_n\otimes I_{m+p})^T(\mathbf{B}\bar{\mathbf{y}}^k-\mathbf{c})\|^2\nonumber\displaybreak[0]\\
=&\|\sum_{i \in \mathcal{V}}\bar{t}_i^k\|^2+\|\sum_{i \in \mathcal{V}}\Bigl(\bigl(\sum_{j \in \mathcal{N}_i}A_{ji}\bar{x}_i^k\bigr)-b_i\Bigr)\|^2\nonumber\displaybreak[0]\\
=&\|\sum_{i \in \mathcal{V}}\bar{t}_i^k\|^2+\|\sum_{i \in \mathcal{V}}(A_{i}\bar{x}_{\mathcal{N}_i}^k-b_i)\|^2,\label{eq:equivequalitysum}
\end{align}
where $\bar{t}_i^k=\frac{1}{k}\sum_{\ell=1}^kt_i^\ell$. This implies $\|\sum_{i \in \mathcal{V}}(A_{i}\bar{x}_{\mathcal{N}_i}^k-b_i)\| \le \|(\mathbf{1}_n\otimes I_{m+p})^T(\mathbf{B}\bar{\mathbf{y}}^k-\mathbf{c})\|$, which, together with \eqref{eq:Bybarcupperbound}, yields \eqref{eq:priequalitybound}. In addition, \eqref{eq:equivequalitysum} also leads to $\|\sum_{i \in \mathcal{V}}\bar{t}_i^k\| \le \|(\mathbf{1}_n\otimes I_{m+p})^T(\mathbf{B}\bar{\mathbf{y}}^k-\mathbf{c})\|$. By combining this, \eqref{eq:Gybarkupperbound}, and \eqref{eq:Bybarcupperbound} with $\sum_{i\in \mathcal{V}} g_i(\bar{x}^{k}_{\mathcal{N}_{i}})\le \sum_{i\in \mathcal{V}} (g_i(\bar{x}^{k}_{\mathcal{N}_{i}})-\bar{t}_i^{k})+\|\sum_{i\in \mathcal{V}} \bar{t}_i^{k}\|\mathbf{1}_p$, we obtain \eqref{eq:globineqconv}. Finally, \eqref{eq:funcvalupperbound} can be directly derived from \eqref{eq:theofuncvalupperbound}.


\bibliographystyle{IEEEtran}
\bibliography{2025couplingref}

\end{document}